\documentclass[a4paper,12pt]{article}
\usepackage{amssymb,amsmath,amsthm,times,mathrsfs,fullpage}
\usepackage[pdftex]{hyperref}
\hypersetup{plainpages=True, pdfstartview=FitH, bookmarksopen=true,
colorlinks=true,linkcolor=blue,citecolor=blue}
\usepackage{tikz,pgf}
\usepackage{epstopdf}
\usepackage{enumerate}
\usepackage{array}
\usepackage{cite}
\usepackage{booktabs}
\usepackage{lscape}
\usepackage{pdflscape} 

\usepackage{todonotes}

\newcommand{\Ai}{\text{\normalfont Ai}}

\newcommand{\claw}{\ensuremath{\xrightarrow{\mathcal{L}}}}

\def\pf{\noindent \emph{Proof.}\ }

\providecommand{\keywords}[1]
{
  \small	
  \noindent
  \textbf{{Keywords:}} #1
}

\begin{document}

\newtheorem{thm}{Theorem}[section]
\newtheorem{cor}[thm]{Corollary}
\newtheorem{lmm}[thm]{Lemma}
\newtheorem{conj}[thm]{Conjecture}
\newtheorem{pro}[thm]{Proposition}
\newtheorem{df}[thm]{Definition}
\theoremstyle{remark}
\newtheorem{rem}[thm]{Remark}

\newcommand*\samethanks[1][\value{footnote}]{\footnotemark[#1]}
\newcommand*{\myand}{\,\and\,}

\title{\textbf{Enumerative and Distributional Results for $d$-combining Tree-Child Networks}}
\author{Yu-Sheng Chang\thanks{Department of Mathematical Sciences, National Chengchi University, Taipei 116, Taiwan.}
\myand
Michael Fuchs\samethanks
\myand
Hexuan Liu\thanks{Department of Pure and Applied Mathematics, Graduate School of Fundamental Science and Engineering, Waseda University, 169-8555 Tokyo, Japan.}
\myand
Michael Wallner\thanks{Institute for Discrete Mathematics and Geometry, TU Wien, 1040 Vienna, Austria.}
\myand
Guan-Ru Yu\thanks{Department of Applied Mathematics, National Sun Yat-sen University, Kaohsiung 804, Taiwan.}}
\date{\today}
 \maketitle

\begin{abstract}
Tree-child networks are one of the most prominent network classes for modeling evolutionary processes which contain reticulation events.
Several recent studies have addressed counting questions for {\it bicombining tree-child networks} in which every reticulation node has exactly two parents.
We extend these studies to {\it $d$-combining tree-child networks} where every reticulation node has now $d\geq 2$ parents,
and we study one-component as well as general tree-child networks.
For the number of one-component networks, we derive an exact formula from which asymptotic results follow that contain a stretched exponential for $d=2$, yet not for $d \geq 3$.
For general networks, we find a novel encoding by words which leads to a recurrence for their numbers.
From this recurrence, we derive asymptotic results which show the appearance of a stretched exponential for all $d \geq 2$.
Moreover, we also give results on the distribution of shape parameters (e.g., number of reticulation nodes, Sackin index) of a network which is drawn uniformly at random from the set of all tree-child networks with the same number of leaves.
We show phase transitions depending on $d$, leading to normal, Bessel, Poisson, and degenerate distributions.
Some of our results are new even in the bicombining case.

\thispagestyle{empty}
\end{abstract}



\section{Introduction and Results}

The evolutionary process of, e.g., chromosomes, species, and populations is not always tree-like due to the occurrence of reticulation events caused by meiotic recombination on the chromosome level,
specification and horizontal gene transfer on the species level,
and sexual recombination on the population level. Because of this, {\it phylogenetic networks} have been introduced as appropriate models for reticulate evolution. Studying the properties of these networks is at the moment one of the most active areas of research in phylogenetics; see \cite{HuRuSc,St}.

While algorithmic and combinatorial aspects of phylogenetic networks have been investigated now for a few decades, enumerating and counting phylogenetic networks as well as understanding their ``typical shape'' are relatively recent areas of research; see~\cite[page~253]{St} where such questions are only discussed in one short paragraph. However, the last couple of years have seen a lot of progress on these questions, in particular for the class of  {\it tree-child networks}, which is one of the most prominent subclasses amongst the many subclasses of phylogenetic networks; see \cite{CaZh,DiSeWe,FuGiMa1,FuGiMa2,FuHuYu,FuLiYu,FuYuZh,PoBa}.

Most of the studies on tree-child networks have focused on {\it bicombining tree-child networks} which are tree-child networks where every reticulation event involves exactly two individuals.
The purpose of this paper is to discuss extensions of previous results to {\it multicombining tree-child} networks. More precisely, we focus on {\it $d$-combining tree-child networks} which are tree-child networks whose reticulation events involve exactly $d\geq 2$ individuals. In addition to being interesting combinatorial objects, $d$-combining networks are relevant because their investigation leads to an understanding of the particularity of the bicombining case (which is the most important case in applications). More precisely, we are going to see that the two cases $d=2$ and $d>2$ frequently exhibit a very different behavior.

Before explaining our results, we give precise definitions and fix some notation. We start with the definition of phylogenetic networks.

\begin{figure}
\centering
\includegraphics[scale=0.9]{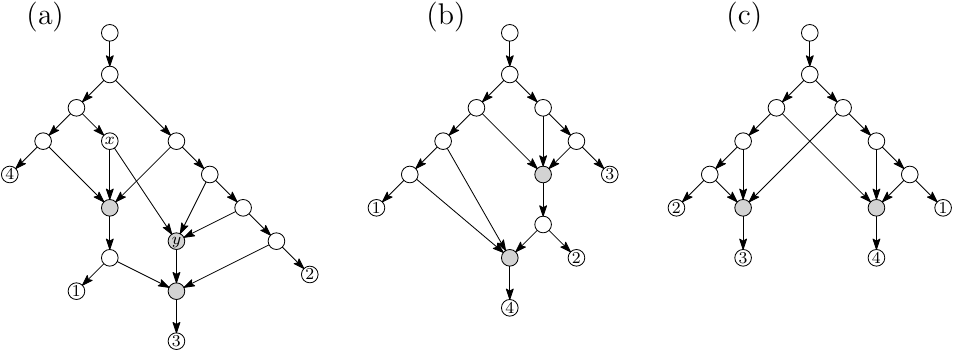}
\caption{(a) A $3$-combining phylogenetic network which is \emph{not} a tree-child network (because both children of the tree node~$x$ are reticulation nodes and the only child of the reticulation node~$y$ is also a reticulation node); (b) a $3$-combining tree-child network; (c) a $3$-combining one-component tree-child network.}\label{pn-fig}
\end{figure}

\begin{df}[Phylogenetic network]
A (rooted) phylogenetic network with $n$ leaves is a  simple, directed acyclic graph (DAG) with no nodes of in- and out-degree~$1$, a (unique) node of in-degree $0$ and out-degree $1$ (the root) and exactly $n$ nodes of in-degree $1$ and out-degree $0$ (the leaves) which are bijectively labeled with labels from the set $\{1,\ldots,n\}$.
\end{df}

This definition is very general. In the sequel, we restrict ourselves to the above mentioned $d$-combining networks where $d\geq 2$ is a fixed integer. ($d=2$ is the bicombining case.)

\begin{df}[$d$-combining network]
A phylogenetic network is a $d$-combining network if all internal nodes (i.e. nodes which are neither leaves nor the root) have either in-degree $1$ and out-degree $2$ ({\it tree nodes}) or in-degree $d$ and out-degree $1$ ({\it reticulation nodes}).
\end{df}

See Figure~\ref{pn-fig} for examples with $d=3$.
We next recall the definition of tree-child networks.

\begin{df}[Tree-child network]
A $d$-combining network is called a tree-child network if every non-leaf node has at least one child which is not a reticulation node.
\end{df}

In other words, a $d$-combining network is a tree-child network if (a) the root is not followed by a reticulation node; (b) a reticulation node is not followed by another reticulation node; and (c) a tree node has at least one child which is not a reticulation node; see Figure~\ref{pn-fig}, (b) for an example.
A simple and important subclass of tree-child networks is the class of one-component tree-child networks; see the definition below and  Figure~\ref{pn-fig}, (c) for an example.

\begin{df}[One-component tree-child network]
A tree-child network is called a one-component tree-child network if every reticulation node is directly followed by a leaf.
\end{df}

One-component networks are more ``tree-like'' than general tree-child networks. Moreover, they constitute an important building block in the construction of general tree-child networks with the component graph method; see \cite{CaZh} for the bicombining case and Appendix~\ref{App-B} for the $d$-combining case. (The name also comes from this method since one-component tree-child networks are those tree-child networks which have only one non-trivial tree-component; for details see Appendix~\ref{App-B}.)

In this paper, we give exact and asymptotic counting results for the number of one-component and general $d$-combining tree-child networks. Moreover, we investigate the number of reticulation nodes (and other parameters) of a {\it random network} where random here means that the network is picked with the uniform distribution. We detail some of our results below and give more results in the subsequent sections.

In order to state our results, we need some notation. We denote by $\mathrm{OTC}^{(d)}_{n,k}$ and $\mathrm{TC}^{(d)}_{n,k}$ the number of one-component and general $d$-combining tree-child networks with $n$ leaves and $k$ reticulation nodes, respectively.
Note that the tree-child property implies that $k\leq n-1$. Thus, the total number of one-component and general $d$-combining tree-child networks, denoted by $\mathrm{OTC}^{(d)}_n$ and $\mathrm{TC}^{(d)}_n$, satisfy
\[
\mathrm{OTC}^{(d)}_n=\sum_{k=0}^{n-1}\mathrm{OTC}^{(d)}_{n,k}\qquad\text{and}\qquad\mathrm{TC}^{(d)}_n=\sum_{k=0}^{n-1}\mathrm{TC}^{(d)}_{n,k}.
\]

Now, we are ready to present our results.
First, for one-component tree-child networks, we extend the exact formula for $\mathrm{OTC}_{n,k}^{(d)}$ for $d=2$ from \cite[Theorem~13]{CaZh} to arbitrary $d\geq 2$ in Theorem~\ref{formula-OTC}. From this extension, we then derive the following asymptotic counting results.

\begin{thm}\label{Th-1}
The following asymptotic equivalences hold for one-component $d$-combining tree-child networks.
\begin{itemize}
\item[(i)] For $d=2$ (bicombining case), we have
\[
\mathrm{OTC}^{(2)}_n\sim \frac{1}{4 \pi \sqrt{ e}} (n!)^2 2^{n} e^{2\sqrt{n}} n^{-9/4}.
\]
\item[(ii)] For $d=3$, we have
\[
\mathrm{OTC}^{(3)}_n
\sim I_1(2)\cdot\mathrm{OTC}^{(3)}_{n,n-1}\sim\frac{I_1(2)\sqrt{3}}{9\pi}(n!)^3\left(\frac{9}{2}\right)^nn^{-3},
\]
where
\[
I_v(a)=\left(\frac{a}{2}\right)^v\sum_{k=0}^{\infty}\frac{1}{k!\Gamma(k+v+1)}\frac{a^{2k}}{4^k}
\]
is the modified Bessel function of the first kind.
\item[(iii)] For $d\geq 4$, we have
\[
\mathrm{OTC}^{(d)}_n\sim\mathrm{OTC}^{(d)}_{n,n-1}\sim\frac{d!}{d^{d-1/2}(2\pi)^{(d-1)/2}}(n!)^{d}\left(\frac{d^d}{d!}\right)^nn^{3(1-d)/2}.
\]
\end{itemize}
\end{thm}

The result for the case $d=2$ is already contained in \cite{FuYuZh}. Note that it is the only case of the three cases above in which we find a stretched exponential in the asymptotics; see~\cite{ElFaWa}.
The above mentioned formula for $\mathrm{OTC}_{n,k}^{(d)}$ also gives the following distributional result for the number of reticulation nodes.

\begin{thm}\label{Th-2}
Let $R^{(d)}_n$ be the number of reticulation nodes of a one-component $d$-combining tree-child network picked uniformly at random from the set of all one-component $d$-combining tree-child networks with $n$ leaves. Then, we have the following limit behavior of $R^{(d)}_n$.
\begin{itemize}
\item[(i)] For $d=2$ (bicombining case), we have the convergence in distribution\footnote{A sequence $(X_n)_{n \geq 1}$ of random variables converges in distribution (or in law) to a random variable $X$ (denoted by $X_n \claw X$) if $\lim_{n \to \infty} F_{X_n }(x)=F_X(x)$ for each continuity point $x \in \mathbb{R}$ of $F_X$, where $F_{X_n}$ and $F_X$ are the respective cumulative distribution functions.} result:
\[
\frac{R^{(2)}_n-n+\sqrt{n}}{\sqrt[4]{n/4}} \claw N(0,1),
\]
where $N(0,1)$ denotes the standard normal distribution.
\item[(ii)] For $d=3$, we have the convergence in distribution result:
\[
n-1-R^{(3)}_n \claw \operatorname{Bessel}(1,2),
\]
where Bessel$(v,a)$ denotes the Bessel distribution, which is defined via $I_{v}(\alpha)$ from Theorem~\ref{Th-1}~(ii):
\[
\mathbb{P}(\operatorname{Bessel}(1,2)=k)=\frac{1}{I_1(2)k!(k+1)!},\qquad (k\geq 0).
\]
\item[(iii)] For $d\geq 4$, the limit law of $n-1-R_n^{(d)}$ is degenerate at $0$, i.e., $n-1-R^{(d)}_n \claw \text{Dirac}(0)$, where $\text{Dirac}(\lambda)$ denotes the Dirac measure at $\lambda$.
\end{itemize}
\end{thm}

Note that the above result for $d=2$ is already contained in the proof of \cite[Theorem~3]{FuYuZh} where even a local limit theorem was proved; see also \cite{FuLiYu}.

\begin{rem}
If $t$ denotes the number of tree nodes and $\tilde{n}$ the total number of nodes, then by the handshaking lemma, we have
\begin{equation}\label{tree-nodes}
t=n+(d-1)k-1
\end{equation}
and thus,
\[
\tilde{n}=2n+dk.
\]
Therefore, we have similar limit distribution results for these numbers as well.
\end{rem}

We next turn to general tree-child networks. Here, in contrast to one-component tree-child networks, we do not have an closed form for $\mathrm{TC}^{(d)}_{n,k}$. However, we introduce a way of encoding these networks by words and this encoding leads to a recursive method for computing $\mathrm{TC}^{(d)}_{n,k}$. Using this method, the values of this sequence for small $n,k$ and $d$ can be computed; see Appendix~\ref{App-A}.

In addition, we are going to see that the growth of $\mathrm{TC}^{(d)}_{n,k}$ is dominated by $\mathrm{TC}^{(d)}_{n,n-1}$. For the latter sequence, a recurrence used for the computation of its values and the method of \cite{ElFaWa} (which needs some adaptions because of the dependence on $d$) yields the following asymptotic counting result for $\mathrm{TC}^{(d)}_{n}$.
For the bicombining case, this result was proved in \cite{FuYuZh}.
Contrary to the one-component case from Theorem~\ref{Th-1}, in the general case the stretched exponential appears for all $d \geq 2$.

\begin{thm}\label{gen-tc-max-k}
For the number of $d$-combining tree-child networks with $n$ leaves, we have
\begin{equation}\label{asymp-TCn}
\mathrm{TC}^{(d)}_{n}
=\Theta\left(\mathrm{TC}^{(d)}_{n,n-1}\right)
=\Theta\left((n!)^{d} \, \gamma(d)^n \, e^{3a_1\beta(d)n^{1/3}} n^{\alpha(d)} \right),
\end{equation}
where $a_1=-2.33810741\ldots$ is the largest root of the Airy function of the first kind, defined as the unique function $\Ai(z)$ satisfying $\Ai''(z) = z \Ai(z)$ such that $\lim_{z \to \infty} \Ai(z)=0$,  and
\[
\alpha(d)=-\frac{d(3d-1)}{2(d+1)},
\qquad\beta(d)=\left(\frac{d-1}{d+1}\right)^{2/3},
\qquad\gamma(d)=4\frac{(d+1)^{d-1}}{(d-1)!}.
\]
\end{thm}
\begin{table}[t]
\centering
\renewcommand{\arraystretch}{1.2}
\begin{tabular}{r @{\hskip 5\tabcolsep} cr @{\hskip 5\tabcolsep} cr @{\hskip 5\tabcolsep} cr}
\toprule
\multicolumn{1}{c@{\hskip 5\tabcolsep}}{$d$} & $\alpha(d)$ & \multicolumn{1}{c@{\hskip 5\tabcolsep}}{$\approx$} & $\beta(d)$ & \multicolumn{1}{c@{\hskip 5\tabcolsep}}{$\approx$} & $\gamma(d)$ & \multicolumn{1}{c}{$\approx$} \\
\midrule
$ 2$ & $  -{\frac{5}{3}}$ & $-1.67$ & $(  {\frac{1}{3}})^{2/3}$ & $0.48$ & $                        12$ & $   12.00$ \\
$ 3$ & $              -3$ & $-3.00$ & $(  {\frac{1}{2}})^{2/3}$ & $0.63$ & $                        32$ & $   32.00$ \\
$ 4$ & $ -{\frac{22}{5}}$ & $-4.40$ & $(  {\frac{3}{5}})^{2/3}$ & $0.71$ & $           {\frac{250}{3}}$ & $   83.33$ \\
$ 5$ & $ -{\frac{35}{6}}$ & $-5.83$ & $(  {\frac{2}{3}})^{2/3}$ & $0.76$ & $                       216$ & $  216.00$ \\
$ 6$ & $ -{\frac{51}{7}}$ & $-7.29$ & $(  {\frac{5}{7}})^{2/3}$ & $0.80$ & $        {\frac{16807}{30}}$ & $  560.23$ \\
$ 7$ & $ -{\frac{35}{4}}$ & $-8.75$ & $(  {\frac{3}{4}})^{2/3}$ & $0.83$ & $        {\frac{65536}{45}}$ & $ 1456.36$ \\
$ 8$ & $ -{\frac{92}{9}}$ & $-10.22$ & $(  {\frac{7}{9}})^{2/3}$ & $0.85$ & $      {\frac{531441}{140}}$ & $ 3796.01$ \\
\bottomrule
\end{tabular}
\caption{Specific values of the asymptotic parameters $\alpha(d)$, $\beta(d)$, and $\gamma(d)$ from Theorem~\ref{gen-tc-max-k}.}
\label{tab:gen-tc-max-k}
\end{table}

The first few specific values of the asymptotic parameters $\alpha(d)$, $\beta(d)$, and $\gamma(d)$ 
are shown Table~\ref{tab:gen-tc-max-k}.
Also, by performing a finer analysis for $k$ close to $n$, we obtain the following distributional result for the number of reticulation nodes. (In the case $d\geq 3$, the above mentioned encoding again plays an important role in the proof; the case $d=2$ needs a different treatment.)

\begin{thm}\label{ll-gen-tc}
Let $T_n^{(d)}$ be the number of reticulation nodes of a $d$-combining tree-child network picked uniformly at random from the set of all $d$-combining tree-child networks with $n$ leaves. Then, we have the following limit behavior of $T_n^{(d)}$.

\begin{itemize}
\item[(i)] For $d=2$ (bicombining case), we have the convergence in distribution result:
\[
n-1-T_n^{(2)} \claw \operatorname{Poisson}(1/2),
\]
where $\operatorname{Poisson}(\alpha)$ denotes the Poisson distribution.
\item[(ii)] For $d\geq 3$, the limit distribution of $n-1-T_n^{(d)}$ is degenerate at $0$.
\end{itemize}
\end{thm}

This result is new even in the case $d=2$. In fact, as far as we are aware of, it constitutes the first limit law result for a shape parameter in random
tree-child networks.
Also, the result for $d=2$ improves \cite[Theorem 1.5, (iii)]{DiSeWe}, which states that the number of reticulation nodes of almost all bicombining tree-child networks with $n$ leaves is asymptotic to $n$.

In addition, the above result can be used to improve and extend \cite[Proposition 1.6, (ii)]{DiSeWe}, which was concerned with the number of {\it twigs} of (bicombining) tree-child networks. A twig is a tree node which is contained in a {\it pendant subtree},
i.e., a tree node that has no reticulation node as descendant.
In \cite{DiSeWe}, it was proved that the number of twigs in a random bicombining tree-child network is $o(n)$. In fact, twigs are even rarer than that.

\begin{cor}\label{cor-1}
Let $W_n^{(d)}$ be the number of twigs of a $d$-combining tree-child network picked uniformly at random from the set of all $d$-combining tree-child networks with $n$ leaves. 
\begin{itemize}
\item[(i)] For $d=2$ (bicombining), we have
\[
\mathbb{E}(W_n^{(d)})={\mathcal O}(1).
\]
\item[(ii)] For $d\geq 3$, the limit law of $W_n^{(d)}$ is degenerate at $0$. More precisely,
\[
\mathbb{E}(W_n^{(d)})\longrightarrow 0.
\]
\end{itemize}
\end{cor}

This result also shows that the expected number of {\it cherries}, i.e., tree nodes with both children leaves, of a random $d$-combining tree-child network is bounded, too (since cherries are clearly twigs).
Note that the number of cherries is a popular parameter in phylogenetics and has been extensively studied for {\it phylogenetic trees} (which are bicombining networks without reticulation nodes).

Finally, Theorem~\ref{ll-gen-tc} also implies an improvement of the first equality in~\eqref{asymp-TCn}.

\begin{cor}\label{cor-2}
The following asymptotic equivalences hold for $d$-combining tree-child networks.
\begin{itemize}
\item[(i)] For $d=2$ (bicombining case), we have $\mathrm{TC}_n^{(2)}\sim\sqrt{e}\cdot\mathrm{TC}_{n,n-1}^{(2)}$.
\item[(ii)] For $d\geq 3$, we have $\mathrm{TC}_n^{(d)}\sim\mathrm{TC}_{n,n-1}^{(d)}$.
\end{itemize}
\end{cor}

\begin{rem}
Note that even with the above result, it is still not possible to give the first-order asymptotics of $\mathrm{TC}_n^{(d)}$ since the approach of \cite{ElFaWa} (which we are going to use in order to prove Theorem~\ref{gen-tc-max-k}) gives only a Theta-result.
\end{rem}

We next give an outline of the paper. We split our subsequent considerations into two parts: one-component and general $d$-combining tree-child networks.

In Section~\ref{otc}, we consider one-component $d$-combining tree-child networks. This section is divided into three subsections. In the first (Section~\ref{exact-count}), we state and prove an exact counting formula; the second subsection (Section~\ref{numb-ret-otc}) then uses this formula to deduce Theorem~\ref{Th-1} (asymptotic counting) and Theorem~\ref{Th-2} (limit law for the number of reticulation nodes). Finally, in Section~\ref{Sackin}, we discuss another shape parameter, namely, the Sackin index. This index was recently investigated in \cite{Zh} where the order of the mean was found for the bicombining case. In the current paper, we simplify the analysis and extend it to the $d$-combining case.

In Section~\ref{gen-net}, we present our results for general $d$-combining tree-child networks. Again this section is split into three subsections. In Section~\ref{enc-words}, we show how to encode networks by words that satisfy certain restrictions. This encoding is similar to the one proposed in \cite{PoBa}; however, our encoding offers two advantages: (a) it can be rigorously proved (the encoding in \cite{PoBa} is just conjectural) and (b) our encoding works for all $d$ (whereas the encoding from \cite{PoBa} seems to be restricted to $d=2$). As mentioned above, this encoding leads to a recursive method for computing $\mathrm{TC}_{n,k}^{(d)}$. This method is then used in Section~\ref{asymp-count} to derive the Theta-result of Theorem~\ref{gen-tc-max-k}. In addition, the encoding is also helpful in Section~\ref{numb-ret-tc} which contains the proof of Theorem~\ref{ll-gen-tc} (limit law for the number of reticulation nodes) in Section~\ref{ll-d-2} (for $d=2$) and Section~\ref{ll-d-larger-2} (for $d\geq 3$). Moreover, in Section~\ref{sec:proofcor-1} we prove Corollary~\ref{cor-1} for the number of twigs.

We finish the paper with some concluding remarks and comment on further interesting generalizations in Section~\ref{con}.
Additionally, the paper contains three appendices:
In Appendix~\ref{App-A}, we list values of $\mathrm{TC}_{n,k}^{(d)}$ for small values of $n,k,d$. These values are computed with the recurrence from Section~\ref{enc-words}. Alternatively, the computation can be done with the method of component graphs from \cite{CaZh} whose extensions to $d$-combining networks is explained in Appendix~\ref{App-B}. However, the computation of values via this method is more involved. Nevertheless, the method is  still useful since it allows us to find the first-order asymptotics of $\mathrm{TC}_{n,k}^{(d)}$ as $n\rightarrow\infty$ and fixed $k$; see \cite{FuGiMa1,FuGiMa2,FuHuYu} for the corresponding results in the bicombining case. Moreover, the method of component graphs also gives formulas for $\mathrm{TC}_{n,k}^{(d)}$ for (fixed) small values of $k$ and $d$; see \cite{CaZh} and \cite{FuGiMa2} for similar formulas for $d=2$. We list some of these formulas in Appendix~\ref{App-B}.
Finally, in Appendix~\ref{App-C}, we give the technical proofs of two asymptotic inequalities needed in the proof of Theorem~\ref{gen-tc-max-k}.

\smallskip

We conclude the introduction by explaining the difference between this paper and the extended abstract \cite{ChFuLiWaYu} which was presented at the {\it 33rd International Meeting on Probabilistic, Combinatorial and Asymptotic Methods for the Analysis of Algorithms (AofA2022)} that took place from June 20th to June 25th at the University of Pennsylvania, Philadelphia, USA. For one-component networks, the material of Section~\ref{exact-count} and Section~\ref{numb-ret-otc} was already contained in \cite{ChFuLiWaYu}, while the material from Section~\ref{Sackin} is new. For general networks, \cite{ChFuLiWaYu} just contained a sketch of the proof of Theorem~\ref{gen-tc-max-k}. Here, we give the full proof in Section~\ref{asymp-count}. Regarding Sections~\ref{enc-words} and~\ref{numb-ret-tc}, some of its results were just mentioned (without any details) in the conclusion of \cite{ChFuLiWaYu}, e.g., Theorem~\ref{ll-gen-tc} (which we prove in Section~\ref{numb-ret-tc}) was stated only as conjecture. Finally, in Appendix~\ref{App-B}, we prove and correct \cite[Theorem~8]{ChFuLiWaYu}, which was not proved previously. Also, the formulas in this appendix have been announced in \cite{ChFuLiWaYu}.

\section{One-Component Networks}\label{otc}

In this section, we prove results for one-component $d$-combining tree-child networks.

\subsection{Exact Counting}\label{exact-count}

In \cite{CaZh}, the authors derived an exact formula for $\mathrm{OTC}_{n,k}^{(2)}$. We give now the generalization of this formula to $d$-combining networks.
Since both Theorem~\ref{Th-1} and Theorem~\ref{Th-2} are derived from this result, we give two different proofs of it; one below and one in Remark~\ref{rem:OTCprooftwo} succeeding the proof.
The first proof gives a direct combinatorial interpretation of the closed form,
while the second one uses a recursion on $d$ and eventually relies on the formula for $\mathrm{OTC}_{n,k}^{(2)}$ in \cite{CaZh}.

\begin{thm}\label{formula-OTC}
The number of one-component $d$-combining tree-child networks with $n$~leaves and $k$~reticulation nodes for $0 \leq k \leq n-1$ is given by
\[
\mathrm{OTC}^{(d)}_{n,k}=\binom{n}{k}\frac{(2n+(d-2)k-2)!}{(d!)^k \, 2^{n-k-1} \,(n-k-1)!}.
\]
\end{thm}

\begin{proof}
    We give the closed form for one-component networks a direct combinatorial interpretation.
    For this purpose, we construct all one-component networks as follows:
    \begin{enumerate}[(i)]
        \item \label{rem:OTCdirect-step1}
            Start with a phylogenetic tree with $n-k$ leaves, i.e., without reticulation nodes.

        \item \label{rem:OTCdirect-step2}
            Place $dk$ unary nodes along the $2(n-k)-1$ edges.

        \item \label{rem:OTCdirect-step3}
            Attach $k$ reticulation nodes that are each connected to $d$ of the unary nodes.
            Attach to each reticulation node a leaf and label it from $1$ to $k$ in the order of creation.

        \item \label{rem:OTCdirect-step4}
            Relabel the $n$ leaves respecting the orders of the $n-k$ initial  and $k$ newly created leaves.
    \end{enumerate}
    Each in that way created network is different and all one-component networks satisfy such a decomposition.
    Multiplying the number of possibilities for each step gives the claimed formula:
    \begin{align}
        \label{eq:OTCdirect}
        \mathrm{OTC}^{(d)}_{n,k} &=
            \mathrm{OTC}^{(d)}_{n-k,0} \cdot \binom{2(n-k)+dk-2}{dk} \cdot \binom{dk}{d,d,\dots,d} \cdot \binom{n}{k}.
    \end{align}
    The result follows now by the fact that
    \begin{align}
    \label{eq:phylotreeformula}
    \mathrm{OTC}^{(d)}_{n-k,0}=(2(n-k)-3)!!=\frac{(2n-2k-2)!}{2^{n-k-1}(n-k-1)!}
    \end{align}
    since $\mathrm{OTC}^{(d)}_{n-k,0}$ is the number of phylogenetic trees with $n-k$ leaves; e.g., see \cite[Section 2.1]{St}.
\end{proof}

\begin{rem}\label{rem:OTCprooftwo} A second way to obtain $\mathrm{OTC}_{n,k}^{(d)}$ proceeds by constructing $\mathrm{OTC}_{n,k}^{(d)}$ from $\mathrm{OTC}_{n,k}^{(d-1)}$ as follows.
First, let a $(d-1)$-combining one-component tree-child network with $n$ leaves and $k$ reticulation nodes be given.
Then, there are  $2n+(d-3)k-1$ edges whose end points are not reticulation nodes (i.e. candidate edges).
We add $k$ different nodes (each node being the $d$th parent of a reticulation node) to these edges. Overall there are
\[
\prod_{i=0}^{k-1} (2n+(d-3)k-1 + i)
\]
ways to do this.
Now, if we assign the first one of the $k$ nodes to the first reticulation node, the second one to the second reticulation node, etc., we obtain every one-component $d$-combining tree-child network with $n$ leaves and $k$ reticulation nodes exactly $d^k$ times. Thus,
\[
\frac{\mathrm{OTC}_{n,k}^{(d)}}{\mathrm{OTC}_{n,k}^{(d-1)}}=
\frac{%
\prod_{i=0}^{k-1} (2n+(d-3)k-1 + i)
}{d^k}.
\]
From this, we can get the result for one-component $d$-combining networks by iteration and using the known result for the bicombining case from \cite{CaZh}.
\end{rem}

\subsection{Number of Reticulation Nodes}\label{numb-ret-otc}

From Theorem~\ref{formula-OTC}, we can now deduce Theorems~\ref{Th-1} and~\ref{Th-2} by the Laplace method. (A standard method of asymptotic analysis; see, e.g., \cite[Chapter~4.7]{FlSe}.)

\vspace*{0.35cm}\noindent{\it Proof of Theorems~\ref{Th-1} and~\ref{Th-2}.} Since the results for $d=2$ are already contained in \cite{FuYuZh} (see also \cite{FuLiYu}), we can focus on the cases $d\geq 3$.
We start with the case $d=3$. Note that
\[
\mathrm{OTC}^{(3)}_{n,k}=\binom{n}{k}\frac{(2n+k-2)!}{3^k2^{n-1}(n-k-1)!},\qquad (0\leq k\leq n-1)
\]
and this sequence is increasing in $k$. (This is in contrast to $d=2$ where this sequence increases until its maximum at $k=n-\sqrt{n+1}$ and then decreases; see \cite{FuYuZh}.) By replacing $k$ by $n-1-k$ and using Stirling's formula, we obtain
\begin{equation}\label{exp-otc}
\mathrm{OTC}^{(3)}_{n,n-1-k}=\frac{1}{k!(k+1)!}\cdot\frac{n(3n-3)!}{6^{n-1}}\left(1+{\mathcal O}\left(\frac{1+k^2}{n}\right)\right)
\end{equation}
uniformly for $k$ with $k=o(\sqrt{n})$. Thus, by a standard application of the Laplace method, we get
\begin{align*}
\mathrm{OTC}_n^{(3)}
\sim\bigg(\sum_{k\geq 0}\frac{1}{k!(k+1)!}\bigg)\cdot\frac{n(3n-3)!}{6^{n-1}}=I_1(2)\cdot\frac{n(3n-3)!}{6^{n-1}},
\end{align*}
which is the first claim from Theorem~\ref{Th-1}, (ii); the second follows from this by another application of Stirling's formula. Moreover, since
\[
\mathbb{P}(R_n^{(3)}=n-1-k)=\frac{\mathrm{OTC}^{(3)}_{n,n-1-k}}{\mathrm{OTC}_n^{(3)}},
\]
the result from Theorem~\ref{Th-2}, (ii) follows from the above two expansions, too.

Next, we consider the case $d\geq 4$.
The details of the proof are the same as above, with the main difference that the expansion~\eqref{exp-otc} now becomes
\[
\mathrm{OTC}^{(d)}_{n,n-1-k}=\left(\frac{d^2d!}{2d^d}\right)^k\frac{1}{k!(k+1)!}\cdot n^{(3-d)k}\cdot\frac{n(dn-d)!}{(d!)^{n-1}}\left(1+{\mathcal O}\left(\frac{1+k^2}{n}\right)\right)
\]
uniformly for $k$ with $k=o(\sqrt{n})$.
For $d\geq 4$ this expansion contains the (non-trivial decreasing) factor $n^{(3-d)k}$,
and therefore $\mathrm{OTC}_n^{(d)}$ is now asymptotically dominated by $\mathrm{OTC}^{(d)}_{n,n-1}$ (proving Theorem~\ref{Th-1}, (iii)) and the limiting distribution of $n-1-R_n^{(d)}$ is degenerate (proving Theorem~\ref{Th-2}, (iii)).
\qed

\subsection{Sackin Index}\label{Sackin}

In this section, we investigate another shape parameter for random one-component tree-child networks, namely, a Sackin index. For phylogenetic trees, the investigation of this index has a long history and many results have been proved; see \cite{FiHeKeKuWi,FiHeKeKuWi2} for a summary of some of these results. For tree-child networks, \cite{Zh} recently gave the first generalization of the Sackin index to networks.
We simplify this approach
(which treated the bicombining case) and extend the main results to the $d$-combining case.

Until the end of this subsection, we assume that $N$ is a one-component tree-child network with $n$ leaves and $k$ reticulation nodes.
Let us first state the definition of the Sackin index from \cite{Zh}.

\begin{df}[Sackin index]
The Sackin index of $N$, in symbols $S(N)$, is defined as the sum over all leaves of the lengths of the longest paths to the leaves.
\end{df}

For the analysis, following \cite{Zh}, we define a second index.
The {\it top tree component} of $N$, in symbols $C(N)$, is defined as the tree obtained from $N$ by deleting all reticulation nodes together with their incident edges and their leaves below. Note that we retain the parents of all reticulation nodes and thus the top tree component is {\it not} a binary tree, it also has unary nodes.
We denote by $P(N)$ the (total) path length of $C(N)$, i.e., the sum over all root-distances of all vertices. This index is related to the Sackin index as follows.

\begin{lmm}\label{rel-Sackin-P}
For all one-component tree-child networks $N$ with at least two leaves, we have for $d=2$
\[
S(N)\leq P(N)+1\leq 2 S(N)
\]
and for $d\geq 3$
\begin{equation}\label{chain-ineq}
S(N)\leq P(N)\leq d S(N).
\end{equation}

\vspace*{0.1cm}\noindent Consequently, regardless of the value of $d$, $S(N)=\Theta(P(N))$.
\end{lmm}
\begin{proof}
The result for $d=2$ was given in \cite{Zh} and the ideas of the proof in \cite{Zh} can be used to handle also the case $d\geq 3$. For the readers convenience, we give some of the details.

First, we define the following sets of vertices of $N$:
\begin{enumerate}
\item[(i)] $L_{T}$ collects leaves of $N$ which are not below a reticulation node;
\item[(ii)] $L_{R}$ collects leaves of $N$ which are below a reticulation node;
\item[(iii)] $P$ collects the parents of all reticulation nodes;
\item[(iv)] $R=V(C(N))\setminus (L_{T} \cup P)$ collects the remaining vertices in $C(N)$. (Here, $V(C(N))$ denotes the set of vertices of $C(N)$.)
\end{enumerate}

Then,
\[
S(N)=\sum_{v \in L_{T} \cup L_{R}} \mathrm{depth}(v),
\]
where the depth of $v$ is the longest distance to the root.

Now, in order to prove the lower bound of $P(N)$, note that for $v\in L_{R}$, we have
\[
\mathrm{depth}(v)=2+\max\{\mathrm{depth}(p)\ :\ p\ \text{is a grandparent of $v$}\}\leq \sum_{p\  \text{is a grandparent of $v$}}\,\mathrm{depth}(p).
\]
Here, for the last inequality, we used that $d\geq 3$. Thus,
\[
S(N)\leq \sum_{v \in L_{T} \cup P} \mathrm{depth}(v)\leq\sum_{v \in L_{T} \cup P \cup R} \mathrm{depth}(v)=P(N)
\]
which shows the first part of the claim in (\ref{chain-ineq}).

Next, in order to show the upper bound of $P(N)$, we first recall that there exists a bijection $\phi$ from $R$ to $L_{T}$ such that $w$ is an ancestor of $\phi(w)$ for all $w\in R$; see the appendix of \cite{Zh} for details. Consequently,  $\mathrm{depth}(w)\leq \mathrm{depth}(\phi(w))$ and thus,
\begin{align*}
P(N)&\leq \sum_{v \in L_{T}} \mathrm{depth}(v) + d \cdot \sum_{v \in L_{R}} \mathrm{depth}(v) + \underbrace{\sum_{w \in R} \mathrm{depth}(\phi(w))}_{\displaystyle= \sum_{v \in L_{T}} \mathrm{depth}(v)}\\&=2\,\sum_{v \in L_{T}} \mathrm{depth}(v) + d\,\sum_{v\in L_{R}}\mathrm{depth}(l)\leq d\, S(N).
\end{align*}
This completes the proof.
\end{proof}

We next analyze $P(N)$. Denote by $\mathcal{OC}_{n,k}^{(d)}$ the set of all one-component tree-child networks with $n$ leaves and $k$ reticulation nodes where the leaves below the reticulation nodes are labeled by the $k$ {\it largest} labels from $\{1,\ldots,n\}$. Note that
\[
\mathrm{OTC}_{n,k}^{(d)}=\binom{n}{k}\vert\mathcal{OC}_{n,k}^{(d)}\vert,
\]
since all one-component tree-child networks are obtained from the networks in $\mathcal{OC}_{n,k}^{(d)}$ by re-labeling the leaves. Also, set
\[
P(\mathcal{OC}_{n,k}^{(d)}):=\sum_{N\in\mathcal{OC}_{n,k}^{(d)}}P(N).
\]

For this quantity, we obtain a (surprisingly) simple recurrence which then yields an exact formula; see \cite[Theorem~3]{Zh} for the bicombining case.

We first prove this recurrence with a computational approach (which extends and simplifies the one from \cite{Zh}). Then, we give a simpler and more elegant combinatorial proof of the solution of the recurrence in two (long) remarks below the proof (Remark~\ref{rem:PathLengthUnaryBinary} and Remark~\ref{rem:PathLengthToptree}). (This second proof however requires the knowledge of the final result which we found with the first proof.) The reader who is only interested in the combinatorial proof might immediately skip to these remarks as the arguments in the second proof are independent from the arguments in the first proof.


\begin{pro}
\label{pro:SackinOCnk}
$P(\mathcal{OC}_{n,k}^{(d)})$ satisfies the recurrence
\[
P(\mathcal{OC}_{n,k}^{(d)}) = \binom{2n+(d-2)k}{d} P(\mathcal{OC}_{n-1,k-1}^{(d)})
\]
with ${\displaystyle P(\mathcal{OC}_{n,0}^{(d)}) = (2n)!!-(2n-1)!!}$ as initial condition. Consequently,
\begin{equation}\label{PdOC}
    P(\mathcal{OC}_{n,k}^{(d)}) =
    \frac{(2n+(d-2)k)!}{(d!)^{k}(2n-2k)!}\,
    \big((2n-2k)!!-(2n-2k-1)!!\big).
\end{equation}
\end{pro}

\begin{proof}

Let $N$ be a network in $\mathcal{OC}_{n-1,k-1}^{(d)}$. We subsequently denote by $V(C(N))$ and $E(C(N))$ the vertex set and edge set of the top tree component, respectively. Moreover, for a $v\in V(C(N))$, we denote by $\delta_{C(N)}(v)$ the number of descendants and by $\alpha_{C(N)}(v)$ the number of ascendants of $v$ in $C(N)$. Note that the number of ascendants of $v$ in $N$ and $C(N)$ is the same, i.e., $\alpha_{C(N)}(v)=\alpha_N(v)$. Finally, we denote by $\mathcal{S}$ the set of multisets of $d$ edges of $C(N)$ (where edges are counted with repetition). Every $S\in\mathcal{S}$ corresponds to a set of edges where nodes are inserted in order to add an additional reticulation node with label $n$. This notion allows us to construct all networks of  $\mathcal{OC}_{n,k}^{(d)}$ from those of $\mathcal{OC}_{n-1,k-1}^{(d)}$. More precisely, for $S\in\mathcal{S}$, we construct a network $N'=N'(S)$ by inserting nodes into the $d$ edges from $S$ and connecting them with a new reticulation node whose child has label $n$. (Note that this is a similar construction to the one we used in the proof of Theorem~\ref{formula-OTC}.)

We consider now $P(N')$ which by definition is given by:
\[
P(N')=\sum_{v \in V(C(N'))} \alpha_{N'}(v).
\]
Partitioning vertices of $C(N')$ into vertices of $C(N)$ and the $d$ new ones which have been added to $N$ to construct $N'$, we have
\[
P(N')=\sum_{v \in V(C(N))} \alpha_{N'}(v) + \sum_{v \in D} \alpha_{N'}(v),
\]
where $D=V(C(N')) \setminus V(C(N))$ is the set of new vertices.

Now, to find a relation to $P(N)$, we replace $\alpha_{N'}$ by $\alpha_{N}$ and get
\begin{equation}\label{exp-PdN}
P(N')=\sum_{v \in V(C(N))} \alpha_{N}(v)+ \sum_{v \in V(C(N))} \beta(v) + \sum_{v \in D} \alpha_{N}(v^{\downarrow})+ \sum_{v \in D} \gamma(v),
\end{equation}
where
\begin{enumerate}
\item[(i)] $\beta(v)$ denotes the number of nodes in $D$ that are ascendants of $v$ in $C(N')$;
\item[(ii)] $v^{\downarrow}$ is closest descendant of $v$ in $C(N')$ that belongs to $V(C(N))$;
\item[(iii)] $\gamma(v)$ denotes the number of nodes in $D$ that are ascendants of $v$ in $C(N')$.
\end{enumerate}

Observe that the second sum can be rewritten as:
\[
\sum_{v\in V(C(N))}\beta(v)=\sum_{v\in D}\delta_{C(N)}(v^{\uparrow}),
\]
where $v^{\uparrow}$ is the closest ascendant of $v$ in $C(N')$ which belongs to $V(C(N))$. Thus, (\ref{exp-PdN}) can be rewritten into:
\[
P(N')=P(N) +\sum_{v\in D} \left(\delta_{C(N)}(v^{\uparrow})+ \alpha_{N}(v^{\downarrow})\right)+\sum_{v \in D}\gamma(v).
\]

Next, we sum both sides over $S \in \mathcal{S}$, which on the right-hand side gives three sums which we denote by $\Sigma_1, \Sigma_2$ and $\Sigma_3$, respectively.

First, for $\Sigma_1$, note that $E:=\vert E(C(N))\vert=2(n-k)-1+dk$. Thus,
\[
\Sigma_1=\sum_{S\in \mathcal{S}}P(N) =\vert\mathcal{S}\vert P(N)= \binom{E+d-1}{d} P(N)
\]
since
\[
\vert\mathcal{S}\vert=\#\{x_1+\cdots+x_E=d\ : \ x_i\ \text{non-negative integers},\ 1\leq i\leq E\},
\]
where each $x_i$ corresponds to an edge from $E(C(N))$ (and counts how many times that edge occurs in the set $S$).

Next, for $\Sigma_2$, set $S= \{e_1,\ldots, e_d\}$ and 
\[
B(e):= \delta_{C(N)}(v^{\uparrow}) + \alpha_{N}(v^{\downarrow}),
\]
where $e$ is an edge in $S$ into which $v$ is inserted and $v^{\uparrow}$ and $v^{\downarrow}$ are initial and end point of $e$, respectively. Then, \[
\Sigma_2=\sum_{S\in \mathcal{S}} \left( B(e_1) + \cdots + B(e_{d}) \right).
\]
We fix an edge $e$ in $C(N)$ and count the number of times $B(e)$ occurs in the above sum. This gives,
\[
\Sigma_2=\sum_{e\in E(C(N))}\left(\sum_{x_1+\cdots +x_E = d} x_i\right)B(e)=2\, \binom{E+d-1}{d-1}P(N),
\]
where $x_i$ inside the second sum is the term corresponding to edge $e$ and for the third expression, we used that
\[
\sum_{x_1+\cdots+x_E=d}x_i = \frac{1}{E} \sum_{x_1+\cdots+x_E=d}x_1 + \cdots + x_E = \binom{E+d-1}{d-1},
\]
which holds for all $1\leq i\leq E$ by symmetry. Moreover, we used that
\[
\sum_{e \in E(C(N))} B(e) = 2\,P(N).
\]

Finally, for $\Sigma_3$, we first give an explicit formula
\[
\Sigma_3= \sum_{v \in C(N) \setminus \{\rho\}}\sum_{0 \leq i + j \leq d} \left( ij + \frac{i(i-1)}{2}\right)\binom{\alpha_{N}(v)-2+j}{j}\binom{E-\alpha_{N}(v)-1+d-i-j}{d-i-j},
\]
by the following combinatorial steps:
\begin{enumerate}[(i)]
\item Choose a vertex $v$ in $C(N)$ which is not the root vertex $\rho$. We consider the incoming edge $e$ of $v$;
\item Choose pair $(i,j)$ with $0\leq i+j\leq d$. Here, the meaning of $i$ is that $i$ nodes are inserted into $e$ and $j$ nodes are inserted into an edge which lies on the path from $\rho$ to $v$ (excluding $e$);
\item So far, the contribution of $v$ to $\Sigma_3$ is
\[
ij + \sum_{\ell=1}^{i-1}\ell=ij+\frac{i(i-1)}{2};
\]
\item Next, there are $\binom{\alpha_{N}(v)-2+j}{j}$ ways of choosing the edges into which the $j$ nodes are inserted;
\item Finally, the second binomial coefficient counts the number of ways that the remaining $d-i-j$ new nodes are inserted into edges which are not on the path from $\rho$ to $v$.
\end{enumerate}
In order to simplify the formula, we set
\[
\alpha:=\alpha_{N}(v),\qquad A:=\alpha-2,\qquad  B:=E-\alpha-1,\qquad M:=d-i.
\]

\newpage

Then, for the inner sum $\Sigma_3$, we have
\begin{align*}
& \sum_{i=0}^{d} \left( \small i \sum_{j=0}^{M} j \binom{A+j}{j} \binom{B+M-j}{M-j} + \frac{i(i-1)}{2} \sum_{j=0}^{M} \binom{A+j}{j} \binom{B+M-j}{M-j} \right) \\
& =\sum_{i=0}^{d}\left(i(\alpha - 1) [z^{M-1}] (1-z)^{-A-2} (1-z)^{-B-1} +  \frac{i(i-1)}{2} [z^M] (1-z)^{-A-1} (1-z)^{-B-1}\right) \\
& = (\alpha - 1) \sum_{i=0}^{d} i \binom{A+B+1+M}{M-1} + \sum_{i=0}^{d} \frac{i(i-1)}{2} \binom{A+B+1+M}{M} \\
& = (\alpha - 1) [z^{d-2}] (1-z)^{-A-B-5} + [z^{d-2}]  (1-z)^{-A-B-5} \\
&= \alpha\,\binom{E+d-1}{d-2}.
\end{align*}
Plugging this into the above formula for $\Sigma_3$ gives
\[
\Sigma_3=\binom{E+d-1}{d-2}\sum_{v\in V(C(N))\setminus\{\rho\}}\alpha_N(v)=\binom{E+d-1}{d-2}P(N).
\]

Finally, summing the above expressions for $\Sigma_1, \Sigma_2$ and $\Sigma_3$ and summing over all $N$ in $\mathcal{OC}_{n-1,k-1}^{(d)}$ gives the surprisingly simple recurrence:
\begin{align}
\label{eq:PDOCrec}
P(\mathcal{OC}_{n,k}^{(d)}) = \binom{2n+(d-2)k}{d}\,P(\mathcal{OC}_{n-1,k-1}^{(d)}).
\end{align}
The initial condition, namely, the formula for $P(\mathcal{OC}_{n,0}^{(d)}) = (2n)!!-(2n-1)!!$ is well-known; see, e.g., \cite{Zh}. From this, \eqref{PdOC} is obtained by iteration. This concludes the proof.
\end{proof}

\begin{rem}
\label{rem:PathLengthUnaryBinary}
A {\it unary-binary tree} is a rooted tree whose nodes are either leaves (out-degree $0$), unary (out-degree $1$), or binary (out-degree $2$).
Then, we observe that the set $C(\mathcal{OC}_{n,k}^{(d)})$ of top trees obtained from all elements of $\mathcal{OC}_{n,k}^{(d)}$ is equal to the set of all unary-binary trees with $n-k$ labeled leaves and $dk$ unary nodes.
Hence, \eqref{PdOC} nearly gives the path length of unary-binary trees, up to an overcount (which is actually a factor) that is related to reticulation nodes.
Now, we adapt the construction of the proof of Theorem~\ref{formula-OTC} to build $\mathcal{OC}_{n,k}^{(d)}$ in order to determine this factor.
Steps~\eqref{rem:OTCdirect-step1} and \eqref{rem:OTCdirect-step2} remain the same,
and step~\eqref{rem:OTCdirect-step4} is not performed as the leaves of reticulation nodes have maximal labels in $\mathcal{OC}_{n,k}^{(d)}$.
It remains to consider step~\eqref{rem:OTCdirect-step3}, which needs to be adapted.
Observe that adding reticulation nodes and unary edges corresponds to the factor $\binom{dk}{d,d,\dots,d}$; see Figure~\ref{fig:Sackin-OTC}.
Thus, dividing~\eqref{PdOC} by this factor, we get that the path length of unary-binary trees with $n-k$ labeled leaves and $dk$ unary nodes is equal to
\begin{align}
    \label{eq:PathLengthUnaryBinary}
    \big( (2(n-k))!!-(2(n-k)-1)!! \big) \cdot \binom{2(n-k)+dk}{dk}.
\end{align}

\begin{figure}
\centering
\includegraphics[scale=0.985]{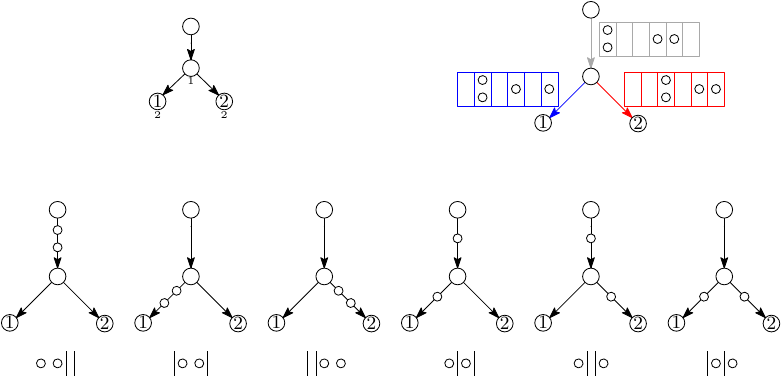}
\caption{%
Construction of $\mathcal{OC}_{3,1}^{(2)}$ in Remark~\ref{rem:PathLengthToptree} ($K=2$, $M=3$).
(Top, left) The only phylogenetic tree with $2$ leaves; path lengths below each node and total path length $5$.
(Bottom) $\binom{4}{2}=6$ top trees $\mathcal{OC}_{2,1}^{(2)}$ created from it after adding $2$ unary nodes and the respective "balls and bars" diagrams.
(Top, right) Superposition of all $6$ top trees.
}
\label{fig:Sackin-OTC}
\end{figure}
\end{rem}

\begin{rem}
\label{rem:PathLengthToptree}
The surprising simplifications of Proposition~\ref{pro:SackinOCnk} also have a direct combinatorial explanation.
By Remark~\ref{rem:PathLengthUnaryBinary}, it suffices to show that the path length of unary-binary trees with $n-k$ labeled leaves and $dk$ unary nodes is equal to~\eqref{eq:PathLengthUnaryBinary}.

The main idea is to use again the construction from the proof of Theorem~\ref{formula-OTC} to build the networks in $\mathcal{OC}_{n,k}^{(d)}$ bottom-up.
We start in step~\eqref{rem:OTCdirect-step1} with a phylogenetic tree (i.e., without reticulation nodes) with $L:=n-k$ leaves that is weighted by the path length.
In the sequel, we call the nodes/edges of this phylogenetic tree, the original nodes/edges.
As noted before, the total weight of such trees is $P(\mathcal{OC}_{L,0}^{(d)}) = (2L)!!-(2L-1)!!$.
Then, we place $K:=dk$ unary nodes along the $M:=2L-1$ edges.
This gives a new structure enumerated by
\begin{align}
    \label{eq:unarybinarystep1}
    P(\mathcal{OC}_{L,0}^{(d)})\binom{K+M-1}{K}.
\end{align}
Note that this is nearly equal to what we want to prove in~\eqref{eq:PathLengthUnaryBinary}, except that the binomial coefficient should be~$\binom{K+M+1}{K}$.
What remains to be done, is to properly change the weights, as they do not correspond to the path lengths anymore, due to the additional unary nodes.

Let us start with a simple observation:
Set all edge weights to $w$ in a phylogenetic tree with $L$ leaves.
Then the path length is equal to $w$ times the path length of the unweighted tree.

Having this observation in mind, we interpret unary nodes as weights on the original edges.
We do this in three steps; the process is visualized in Figures~\ref{fig:Sackin-OTC} and~\ref{fig:Sackin-3steps} for $\mathcal{OC}_{3,1}^{(2)}$.
The main idea is to superimpose all $\binom{K+M-1}{K}$ created trees and group them cleverly.

\begin{figure}
\centering
\includegraphics[scale=0.985]{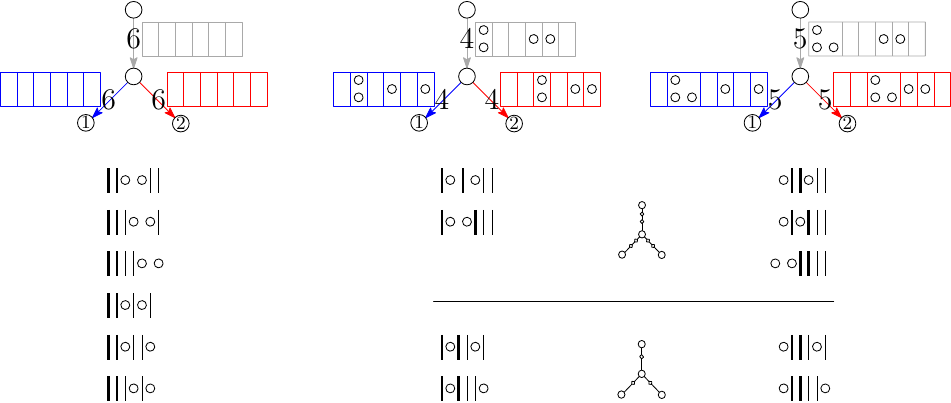}
\vspace*{0.2cm}\caption{%
Enumeration of $P(\mathcal{OC}_{3,1}^{(2)})$ in Remark~\ref{rem:PathLengthToptree} using superposition ($K=2$, $M=3$); see Figure~\ref{fig:Sackin-OTC}.
(Left) Step~1: Each new instance increases edge weight by one; total $\binom{4}{2}=6$.
(Middle) Step~2: Correct path length of original nodes; sum unary nodes per edge.
(Right) Step~3: Add path length for unary nodes; for $i$ nodes weight $1+2+\dots+i$.
(Bottom) "Balls and bars" corresponding to each step; the two added bars are shown bold.
}
\label{fig:Sackin-3steps}
\end{figure}
\vspace*{0.2cm}

First, we ``forget'' about the unary nodes, and interpret the $\binom{K+M-1}{K}$ new instances of a phylogenetic tree as the same phylogenetic tree.
Hence, the weights of all edges change to $\binom{K+M-1}{K}$ and the total number of these weighted phylogenetic trees is~\eqref{eq:unarybinarystep1}.
In the sequel, the following direct interpretation of the binomial coefficient using
``balls and bars'' is useful:
The $K$ unary nodes correspond to balls and the edges $M$ to bins, which are modeled by $M-1$ bars. Then each distribution of unary nodes on the edges corresponds to choosing $K$ out of $M$ bins with repetitions allowed; see Figure~\ref{fig:Sackin-OTC}.

Second, we correct the weights of all original edges.
Note that each unary node splits an edge into two and thereby increases the path length of any node below by one, or, equivalently, increases the weight of the associated original edge by one.
Hence, the weight of an edge increases by the number of unary nodes it contains.
In the superposition of all trees, all edges have the same number of unary nodes, as each edge is chosen equally often.
The weight is equal to the sum of all unary nodes placed onto a fixed edge among the $\binom{K+M-1}{K}$ configurations.
As there are $M$ edges and $K$ unary nodes (all equally distributed) we need a factor of $K/M$. Thus,
\begin{align*}
    P(\mathcal{OC}_{L,0}^{(d)})\binom{K+M-1}{K} \left( 1 + \frac{K}{M}\right)
    = P(\mathcal{OC}_{L,0}^{(d)})\binom{K+M}{K}.
\end{align*}
A direct interpretation of the last formula is as follows:
Fix the root edge. Mark one of the unary nodes of the root edge.
Like this, we create as many instances (with different markers) as there are unary nodes on the root edge in all configurations, i.e., we count the total number of unary nodes on the root edge.
We model this marker, by splitting the root edge just after this marker into two parts.
Hence, we have now $M+1$ edges to distribute $K$ unary nodes.
We interpret the case when the first part of the root edge is empty, as the $\binom{K+M-1}{K}$ created weighted trees from placing unary nodes.
Thus, there are $\binom{K+M}{K}$ many choices.

Third, we assign weights for the new unary nodes.
Consider an original edge with $i$ unary nodes.
To give the lowest one its correct weight, we remove the original node just below and replace it by the lowest unary node.
Then, we may use the idea of the previous step:
The weight of the edge above is now $i$, as $i-1$ unary nodes remain.
In the superposition, each such sequence of $i$ unary nodes appears on each original edge the same number of times.
Hence, we collect a weight $i$ on each edge, and the total weight gives the correct weights for all lowest unary nodes.
Then, we repeat this process for the next unary node, with an edge weight $i-1$, etc.
In total, an original edge with $i$ unary nodes gives rise to a total weight of $1+2+\dots+i$.
For an arbitrary edge, this is now, analogously to before, equal to the sum of unary nodes that are before the marked node from the previous step.
Thus, this gives a factor $K/(M+1)$ and we get
\begin{align*}
    P(\mathcal{OC}_{L,0}^{(d)})\binom{K+M}{K} \left( 1 + \frac{K}{M+1}\right)
    = P(\mathcal{OC}_{L,0}^{(d)})\binom{K+M+1}{K}.
\end{align*}
Note, as before, this has a direct combinatorial interpretation, where we split the first part of the root edge again.
Hence, there are now $M+2$ edges to distribute $K$ unary nodes, as claimed.
Moreover, note that the last formula also gives the path length of unary binary trees with $L$ labeled leaves and $K$ unary nodes ($M=2L-1$); compare with Remark~\ref{rem:PathLengthUnaryBinary}.
\end{rem}

Using Proposition~\ref{pro:SackinOCnk}, we are now ready to complete our analysis. Let $P^{(d)}_n$ be the path length of the top tree component of a random one-component tree-child network with $n$ leaves. Then,
\[
    {\mathbb E}(P^{(d)}_n)=\frac{\sum_{k=0}^{n-1}\binom{n}{k}P(\mathcal{OC}_{n,k}^{(d)})}{\mathrm{OTC}_n^{(d)}}.
\]
Applying the Laplace method to the numerator and using Theorem~\ref{Th-1} gives the following result.
\begin{pro}
\begin{itemize}
\item[(i)] For $d=2$ (bicombining case), we have
\[
{\mathbb E}(P^{(2)}_n)\sim 2 \sqrt{\pi}n^{7/4}.
\]
\item[(ii)] For $d=3$, we have
\[
{\mathbb E}(P^{(3)}_n) \sim \frac{9(\cosh(2)- I_0(2))}{2I_1(2)}\,n^2,
\]
where the constant is approximately $4.19438713$.
\item[(iii)] For $d\geq 4$, we have
\[
{\mathbb E}(P^{(d)}_n)\sim\frac{d^2}{2}\,n^2.
\]
\end{itemize}
\end{pro}
\begin{proof}
We start with the bicombining case ($d=2$). Using the result from the Proposition~\ref{pro:SackinOCnk} yields
\begin{align*}
\sum_{k=0}^{n-1}\binom{n}{k}P(\mathcal{OC}_{n,k}^{(2)})&=\sum_{k=0}^{n-1}\binom{n}{k+1}P(\mathcal{OC}_{n,n-1-k}^{(2)})\\
&=\sum_{k=0}^{n-1}\binom{n}{k+1}\frac{(2n)!}{2^{n-1-k}(2k+2)!}\left(2^{k+1}(k+1)!-\frac{(2k+1)!}{2^kk!}\right).
\end{align*}
We break the last sum $S$ into two sums, i.e., $S=S_1+S_2$ according to the two terms in the bracket. Thus,
\[
S_1=\frac{n!(2n)!}{2^{n-2}}\sum_{k=0}^{n-1}\frac{4^k}{(2k+2)!(n-1-k)!}
\]
and we have a similar expression for $S_2$. Note that the terms inside the sum increase until a positive integer $k^{*}$ with $k^{*}=\sqrt{n}+{\mathcal O}(1)$ and decrease afterwards. Moreover, by using Stirling's formula, we see that
\[
\frac{4^k}{(2k+2)!(n-1-k)!}=\frac{1}{8\pi\sqrt{2e}}n^{-3/4}e^{2\sqrt{n}}e^n n^{-n}e^{-x^2/\sqrt{n}}\left(1+{\mathcal O}\left(\frac{1+\vert x\vert}{\sqrt{n}}+\frac{x^3}{n}\right)\right)
\]
uniformly for $\vert x\vert\leq n^{3/10}$ where $k=\sqrt{n}+x$. Consequently, from a standard application of the Laplace method:
\begin{align*}
\sum_{k=0}^{n-1}\frac{4^k}{(2k+2)!(n-1-k)!}&\sim\frac{1}{8\pi\sqrt{2e}}n^{-3/4}e^{2\sqrt{n}}e^nn^{-n}\int_{-\infty}^{\infty}e^{-x^2/\sqrt{n}}{\rm d}x\\
&=\frac{1}{8\sqrt{2e\pi}}n^{-1/2}e^{2\sqrt{n}}e^nn^{-n}
\end{align*}
and thus,
\[
S_1\sim\frac{n!(2n)!}{2^{n-2}}\cdot\frac{1}{8\sqrt{2e\pi}}n^{-1/2}e^{2\sqrt{n}}e^nn^{-n}\sim\frac{1}{2\sqrt{e}}(2n)!2^{-n}e^{2\sqrt{n}},
\]
where we again used Stirling's formula. Similarly, we can derive the asymptotics of $S_2$ which shows that $S_2$ is of a smaller asymptotic order, i.e., $S_2=o(S_1)$. Consequently, $S\sim S_1$. Finally, dividing by the asymptotics of $\mathrm{OTC}_{n}^{(2)}$ from part (i) of Theorem~\ref{Th-1} and using (once more) Stirling's formula gives the claimed result.

Next, for $d=3$, we first note that
\[
\binom{n}{k} P(\mathcal{OC}_{n,k}^{(3)}) =\binom{n}{k}\,\frac{(2n+k)!}{6^k(2n-2k)!}\,\left(2^{n-k}(n-k)!-\frac{(2n-2k-1)!}{2^{n-k-1}(n-k-1)!}\right)
\]
is increasing in $k$ with $0\leq k\leq n-1$. By replacing $k$ by $n-1-k$ and using Stirling's formula, we obtain that
\[
\binom{n}{k+1}P(\mathcal{OC}_{n,n-1-k}^{(3)})=\left(\frac{4^{k+1}}{(2k+2)!}-\frac{1}{(k+1)!^2}\right)\,\frac{(3n)!}{6^n}\,\left(1+\mathcal{O}\left(\frac{1+k^2}{n}\right)\right),
\]
uniformly for $k$ with $k=o(\sqrt{n})$. Thus, by an another application of the Laplace method:
\begin{align*}
\sum_{k=0}^{n-1} \binom{n}{k+1}P(\mathcal{OC}_{n,n-1-k}^{(3)})\sim\left( \sum_{k \geq 1 }\frac{4^k}{(2k)!} -\frac{1}{k!^2}\right)\,\frac{(3n)!}{6^n}=\bigl(\cosh(2)-I_0(2) \bigr)\,\frac{(3n)!}{6^n},
\end{align*}
where $I_0(2)$ is the modified Bessel function (see Theorem~\ref{Th-1}, (ii)). Dividing by the asymptotics of $\mathrm{OTC}^{(3)}_n$ (again see Theorem~\ref{Th-1}, (ii)), we have
\[
{\mathbb E}(P^{(3)}_n)\sim\frac{9(\cosh(2) - I_0(2))}{2I_1(2)}\,n^2.
\]
This proves the claim in this case. (Note the similarity of this proof to the one of Theorem~\ref{Th-2}, (ii).)

For $d \geq 4$, with similar arguments as in the proof of part (iii) of Theorem~\ref{Th-2}:
\begin{align*}
\sum_{k=0}^{n-1}\binom{n}{k+1}P(\mathcal{OC}_{n,n-1-k}^{(d)})& \sim n\,P(\mathcal{OC}_{n,n-1}^{(d)})=\frac{n\,(dn-d+2)!}{2\,(d!)^{n-1}}.
\end{align*}
Dividing by the asymptotics of $\mathrm{OTC}^{(d)}_n$ from part (iii) in Theorem~\ref{Th-2}, we have
\[
{\mathbb E}(P^{(d)}_n) \sim\frac{d^2}{2}\,n^2,
\]
which proves the claimed result also in this case.
\end{proof}
\begin{rem}
A (slightly) weaker version of the above result for the bicombining case was derived in \cite[Proposition~2]{Zh}.
\end{rem}

Denote by $S^{(d)}_n$ the Sackin index of a random one-component tree-child network with $n$ leaves. Then, by combining the last proposition with Lemma~\ref{rel-Sackin-P}, we obtain the main result of this subsection. (This result for $d=2$ was also the main result of \cite{Zh}.)

\begin{thm}
\begin{itemize}
\item[(i)] For $d=2$ (bicombining case), we have
\[
{\mathbb E}(S_n^{(d)})=\Theta(n^{7/4}).
\]
\item[(ii)] For $d\geq 3$, we have
\[
{\mathbb E}(S_n^{(d)})=\Theta(n^2).
\]
\end{itemize}
\end{thm}

\section{General Networks}\label{gen-net}

In this section, we consider general $d$-combining tree-child networks.

\subsection{Encoding Networks by Words}\label{enc-words}

We start with a formula for $\mathrm{TC}_{n,k}^{(d)}$ which is not as explicit as the formula for $\mathrm{OTC}_{n,k}^{(d)}$ from Theorem~\ref{formula-OTC}. (It, however, leads to an efficient recursive way of computing $\mathrm{TC}_{n,k}^{(d)}$.) This formula uses the number of words that are defined next.

\begin{df}\label{def-words}
Let  $\mathcal{C}_{n,k}^{(d)}$ denote the class of words consisting of the letters $\{\omega_1,\ldots,\omega_n\}$ in which $k$ letters occur $d+1$ times and $n-k$ letters occur $2$ times and which satisfy the following condition:
In every prefix of a word, either a letter has not occurred more than $d-2$ times, or, if it has, then the number of occurrences of $\omega_i$ is at least as large as the number of occurrences of $\omega_j$ for all $j>i$. Here, for the letters appearing only $2$ times, we treat the $0$th, $1$st, and $2$nd occurrence as the $(d-1)$st, $d$th, and $(d+1)$st occurrence, respectively.
\end{df}

\begin{rem}
For $k=n$, we recover the words from \cite[Definition~12]{ChFuLiWaYu} which in turn generalized the words from \cite[Definition~2]{FuYuZh} from the bicombining case.
\end{rem}

\begin{rem}
The words $\mathcal{C}_{n,k}^{(d)}$ can also be encoded by Young tableaux with walls, where a wall between two cells indicates that there are no order constraints between the respective entries; see \cite{BaMaWa,BaWa}.
The $i$th column is associated to the $i$th letter $\omega_i$ and its size is equal to the number of occurrences of $\omega_i$.
Therefore, the corresponding Young tableaux consist of $k$ columns with $d+1$ cells and $n-k$ columns with $2$ cells
placed next to each other in such a way that the top cells are all side-by-side.
We put vertical walls between all cells in rows three to $d+1$.
Finally, the cells are filled in increasing order from left to right and bottom to top with the numbers $\{1,2,\dots,
2n+k(d-1)
\}$.
The bijection is as follows:
Read the words from left to right.
A letter $\omega_i$ at position $j$ indicates that the value $j$ is put into column $i$.

This generalizes the class analyzed in \cite[Section~4]{BaWa} consisting of only one row with walls.
Our asymptotic counting result for $\mathcal{C}_{n,n}^{(d)}$ (Theorem~\ref{gen-tc-max-k}) gives then directly the respective result for rectangular Young tableaux with vertical walls in all but the first two rows of shape $(d+1) \times n$.
\end{rem}

The next result connects tree-child networks and the  words from Definition~\ref{def-words}.

\begin{thm}\label{formula-TC}
Let $c_{n,k}^{(d)}:=\vert\mathcal{C}_{n,k}^{(d)}\vert$. Then,
\[
\mathrm{TC}_{n,k}^{(d)}=\frac{n!}{2^{n-k-1}}c_{n-1,k}^{(d)}.
\]
\end{thm}

\begin{rem}
For $d=2$, \cite{PoBa} proposed an encoding of tree-child networks with $n$ leaves and $k$ reticulation nodes by a (slightly) different class of words.
This encoding led to a similar formula for their numbers.
However, whereas the formula from \cite{PoBa} is just a conjecture, we provide a rigorous proof of our result.
\end{rem}

Before proving the above theorem, we recall some concepts and provide generalizations of results from \cite{FuYuZh}.
First, we call a tree node in a tree-child network {\it free} if both its children are not reticulation nodes. The edges to the two children of a free tree node are called {\it free edges}.

\begin{lmm}\label{free-nodes}
A tree-child network with $n$ leaves and $k$ reticulation nodes has exactly $n-k-1$ free tree nodes and thus $2(n-k-1)$ free edges.
\end{lmm}

Note that exactly the same result holds in the bicombining case; see~\cite[Lemma~1]{FuYuZh}.
For an example demonstrating the last lemma see Figure~\ref{3-comb-ex}.

\begin{figure}[t]
\centering
\includegraphics[scale=1]{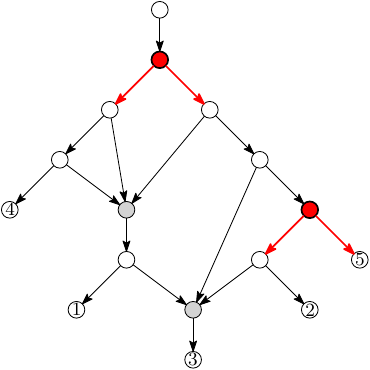}
\caption{A $3$-combining tree-child network with $5$ leaves and $2$ reticulation nodes (gray nodes). Thus, the number of free tree nodes is $2$ (red nodes) and there are $4$ free edges (red edges).}\label{3-comb-ex}
\end{figure}

\begin{proof}
By (\ref{tree-nodes}), the number of tree nodes in a tree-child network with $n$ leaves and $k$ reticulation nodes equals $n+(d-1)k-1$.
Now, observe that the $d$ parents of a reticulation node are
(i) different for all reticulation nodes (because of the tree-child property);
(ii) tree nodes which are not free%
;
and (iii) \emph{all} non-free tree nodes. Thus, the number of free tree nodes equals
\[
n+(d-1)k-1-dk=n-k-1. \qedhere
\]
\end{proof}

We next recall the structure of {\it maximally reticulated tree-child networks}, i.e., tree-child networks with $n$ leaves and $n-1$ reticulation nodes, which was established in \cite[Lemma~4]{FuYuZh} for the bicombining case.
It also carries over to the $d$-combining case:
Every maximally reticulated tree-child network admits a (unique) decomposition into {\it path-components}, which are maximal paths that start at a node and end at a leaf, where all its intermediate nodes are tree nodes. (See Figure~4 in \cite{FuYuZh} for an example.)

Now, we are ready to prove Theorem~\ref{formula-TC}.

\begin{proof}[Proof of Theorem~\ref{formula-TC}]
Let $N$ be a tree-child network with $n$ leaves and $k$ reticulation nodes. (We take the network from Figure~\ref{3-comb-ex} as our running example.) From Lemma~\ref{free-nodes}, we know that $N$ has $n-k-1$ free tree nodes, which each have $2$ free edges. We choose one free edge from each of these pairs of free edges.
Then, the number of all networks $N$ together with choices of free edges equals
\begin{equation}\label{ref-formula}
2^{n-k-1}\mathrm{TC}_{n,k}^{(d)}=n!c_{n-1,k}^{(d)},
\end{equation}
where the equality with the right-hand side is our claim. (See Figure~\ref{bij-TCnk}-(a) for the network from Figure~\ref{3-comb-ex} and its $4$ choices of free edges.)
Thus, in order to prove the claim it suffices to find a bijection between the networks $N$ and a choice of free edges to tuples consisting of a permutation and a word from $\mathcal{C}_{n-1,k}^{(d)}$.

In order to explain this bijection, fix a network $N$ and a choice of free edges; see
Figure~\ref{bij-TCnk}-(a). For every free tree node, insert $d-1$ nodes on its chosen free edge and a single node on its other free edge. Connect the $d-1$ nodes to the single node, thereby turning this single node into a new reticulation node. Notice that the resulting network is a maximally reticulated tree-child network; see Figure~\ref{bij-TCnk}-(b).
The rest of the proof proceeds now as the proof of \cite[Proposition~2]{FuYuZh}.

First, we index the path-components as follows: the path-component containing the root gets index~$0$. Consider all other path-components (which must start with a reticulation node) with parents of the reticulation node already in indexed path components. Index them according to the largest index of the path-component which contains the parents, or, in case of equal largest indices according to whose last parent is encountered first when reading the nodes in the path component from the starting node to the leaf. Repeat this until all path-components have been indexed; see Figure~\ref{bij-TCnk}-(c). Note that one path-component starts with the root, and $n-1$ path-components start with a reticulation node.

Now, label the reticulation node and all its parents of the path-component with index $1$ by $a$, of the path component with index $2$ by $b$, etc.
Next, for each chosen free edge, treat the added $d-1$ nodes (which all have the same label) as a single node; see Figure~\ref{bij-TCnk}-(c).
Then, a word from $\mathcal{C}_{n-1,k}^{(d)}$ is obtained by reading the labels of nodes of the path-components in increasing order.
Finally, a permutation is obtained by reading the labels of each leaf of the path-components in the above order; see Figure~\ref{bij-TCnk}-(d).

Overall, this gives a bijection between $N$ and a fixed choice of free edges for every free tree node to a word from $\mathcal{C}_{n-1,k}^{(d)}$ and a permutation of length $n$. Thus, we have proved (\ref{ref-formula}) and the claim.
\end{proof}

Theorem~\ref{formula-TC} reduces the problem of counting $\mathrm{TC}_{n,k}^{(d)}$ to that of $c_{n,k}^{(d)}$. For the latter sequence, we have the following relation to the sequence $b_{n,k,m}^{(d)}$ which can be computed recursively.

\begin{figure}
\centering
\includegraphics[scale=0.87]{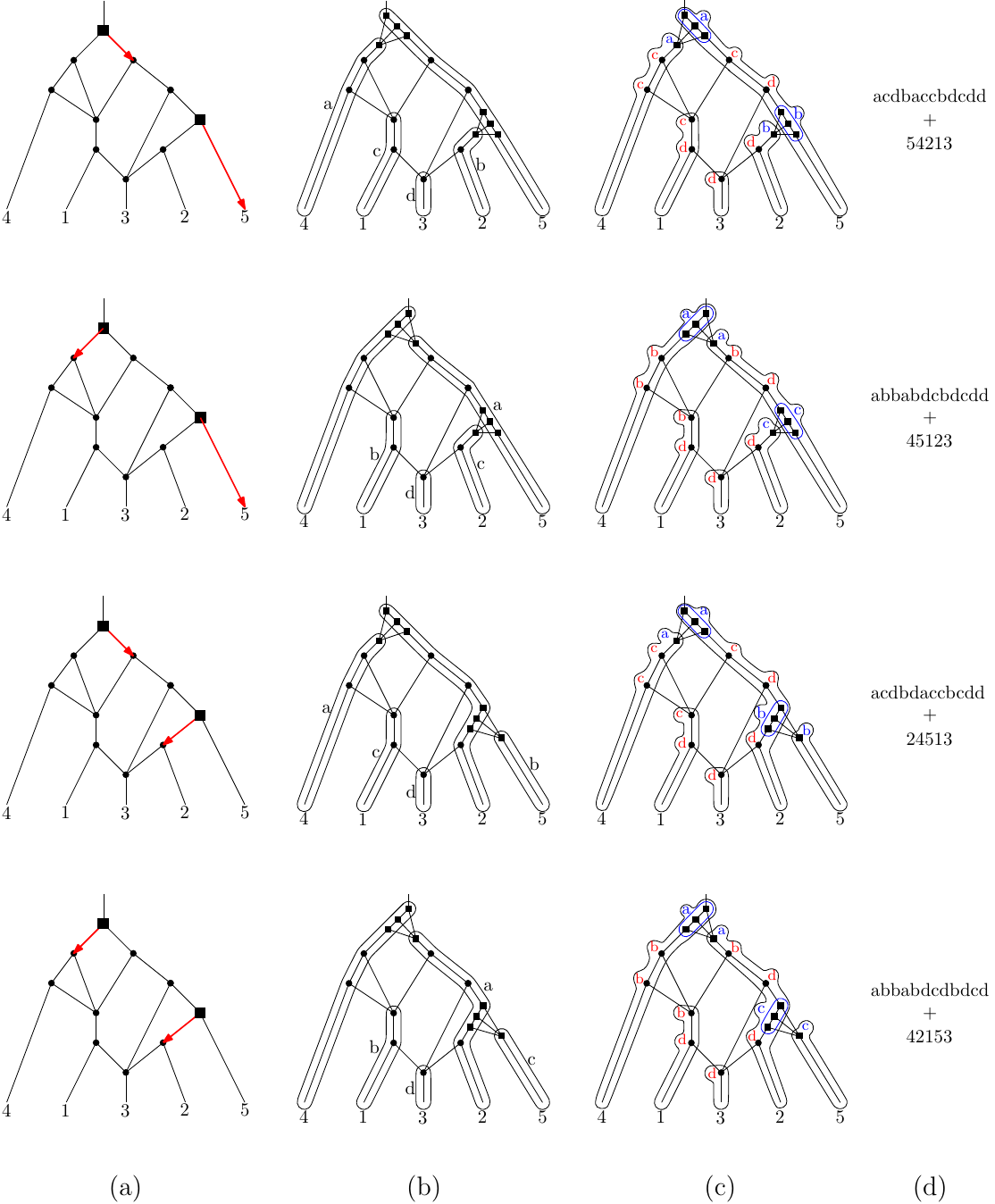}
\caption{(a) The network from Figure~\ref{3-comb-ex} together with the $4$ possible ways of choosing an outgoing free edge for every free tree node; (b) Replacing each free tree node by a reticulation node which results in a maximally reticulated tree-child network whose path-components are indexed; (c) Labeling all internal nodes by labeling reticulation nodes and their parents with the label of their path-component. Note that all nodes on the chosen free edges only receive one label; (d) The word from $\mathcal{C}_{4,2}^{(3)}$ and the permutation corresponding to each network.}\label{bij-TCnk}
\end{figure}

\begin{pro}\label{rec-cnk}
Let $c_{n,k}^{(d)}:=\vert\mathcal{C}_{n,k}^{(d)}\vert$. Then,
\[
c_{n,k}^{(d)}=\sum_{m\geq 1}b_{n,k,m}^{(d)},
\]
where $b_{n,k,m}^{(d)}\ (1\leq m\leq n, 0\leq k\leq n)$ can be recursively computed as
\begin{equation}\label{rec-bnkm}
b_{n,k,m}^{(d)}=\sum_{j=1}^m b_{n-1,k,j}^{(d)}+\binom{n+m+k(d-1)-2}{d-1}\sum_{j=1}^m b_{n-1,k-1,j}^{(d)},\qquad (n\geq 2)
\end{equation}
with initial conditions $b_{n,k,m}^{(d)}=0$ for $n<m$ of $n<k$, $b_{n,-1,m}=0$ and $b_{1,0,1}^{(d)}=b_{1,1,1}^{(d)}=1$.
\end{pro}

\begin{proof}
First, note that because of the condition the words from $\mathcal{C}_{n,k}^{(d)}$ have to satisfy (see Definition~\ref{def-words}), any word from $\mathcal{C}_{n,k}^{(d)}$ has a suffix $\omega_n\omega_m\omega_{m+1}\cdots\omega_{n-1}\omega_n$ with $1\leq m\leq n$.
Denote by $b_{n,k,m}^{(d)}$ the number of these words. We now consider two cases.

First, we assume that $\omega_n$ is a letter which occurs twice. Then, removing the $2$ occurrences of $\omega_n$ from these words gives a word of $\mathcal{C}_{n-1,k}^{(d)}$ with suffix $\omega_m\omega_{m+1}\cdots\omega_{n-1}$, i.e., it has a suffix $\omega_{n-1} \omega_j\omega_{j+1}\dots \omega_{n-1}$ for $j=1,\dots,m$.
Reversing this procedure gives the contribution
\begin{align}
\sum_{j=1}^{m}b_{n-1,k,j}^{(d)}
\end{align}
to $b_{n,k,m}^{(d)}$, which is the first term on the right-hand side of \eqref{rec-bnkm}.

Second, if $\omega_n$ is a letter which occurs $d+1$ times, we remove the $d+1$ occurrences of this letter. Then, with the same line of reasoning as above, we obtain the second term of the right-hand side of (\ref{rec-bnkm}):
\[
\binom{n+m+k(d-1)-2}{d-1}\sum_{j=1}^{m}b_{n-1,k-1,j}^{(d)},
\]
where the binomial coefficient counts the number of ways of adding back the $d-1$ occurrences of $\omega_n$ after two $\omega_n$'s have been added, one before the last $\omega_m$ and one at the end of the word. By Definition~\ref{def-words} these first $d-1$ occurrences of $\omega_n$ may be anywhere.
\end{proof}

The recurrence from the above proposition combined with Theorem~\ref{formula-TC} allows one to compute values of $\mathrm{TC}_{n,k}^{(d)}$ for small values of $n,k,d$; see Appendix~\ref{App-A}. Also, like this, we can recover the table for $\mathrm{TC}_{n,k}^{(2)}$ from \cite{CaZh} which was computed in that paper with a much more computation-intensive approach (which is explained in Appendix~\ref{App-B}).

\subsection{Asymptotic Counting}\label{asymp-count}

In this section, we prove Theorem~\ref{gen-tc-max-k}, namely, we derive an asymptotic result for $\mathrm{TC}_n^{(d)}$. Recall that
\[
\mathrm{TC}_{n}^{(d)}=\sum_{k=0}^{n-1}\mathrm{TC}_{n,k}^{(d)}.
\]

The first main observation is that the last term in this sum dominates; see the first equality in (\ref{asymp-TCn}). More precisely, we have the following.

\begin{lmm}\label{main-obs-TC}
For $0\leq k\leq n-2$, we have
\begin{equation}\label{tc-number-ub}
\mathrm{TC}_{n,k}^{(d)}\leq\frac{1}{2(n-k-1)}\mathrm{TC}_{n,k+1}^{(d)}
\end{equation}
and consequently,
\[
\mathrm{TC}^{(d)}_n=\Theta\left(\mathrm{TC}^{(d)}_{n,n-1}\right).
\]
\end{lmm}

\begin{proof}
Let $N$ be a tree-child network with $n$ leaves and $k$ reticulation nodes. Recall that $N$ has $2(n-k-1)$ free edges; see Lemma~\ref{free-nodes}.

We can construct tree-child networks with $n$ leaves and $k+1$ reticulation nodes from $N$ by (i) inserting $d-1$ tree nodes into the root edge of $N$ and a reticulation node into a free edge and (ii) connecting the $d-1$ new tree nodes to the new reticulation node. Note that each network built in this way is different. Thus,
\[
2(n-k-1)\mathrm{TC}^{(d)}_{n,k}\leq\mathrm{TC}^{(d)}_{n,k+1},
\]
which implies the first claim.

Next, by iteration of (\ref{tc-number-ub}), we obtain
\begin{equation}\label{TCnk-upper-bound}
\mathrm{TC}^{(d)}_{n,k}\leq\frac{1}{2^{n-k-1}(n-k-1)!}\mathrm{TC}^{(d)}_{n,n-1}
\end{equation}
and thus,
\[
\mathrm{TC}^{(d)}_{n,n-1}\leq\mathrm{TC}^{(d)}_{n}\leq\bigg(\sum_{j\geq 0}\frac{1}{2^j j!}\bigg)\cdot\mathrm{TC}^{(d)}_{n,n-1}=\sqrt{e}\cdot\mathrm{TC}^{(d)}_{n,n-1},
\]
which proves the second claim.
\end{proof}

 \begin{rem}\label{eq-max-ret}
Note that for $d=2$ and $k=n-2$, equality holds in (\ref{tc-number-ub}) because in this case, (a) the networks constructed from $N$ in the above proof are maximally reticulated and (b) the child of the root of each maximally reticulated network is not free. Thus, the construction from the proof is reversible and we have a bijection. (This was first proved in \cite[Proposition 17]{CaZh}.)
\end{rem}

As a consequence of the last result, we can now entirely concentrate on the maximal reticulated case for which we obtain from Theorem~\ref{formula-TC}:
\[
\mathrm{TC}_{n,n-1}^{(d)}=n!c_{n-1}^{(d)},
\]
where we have set $c_{n}^{(d)}:=c_{n,n}^{(d)}$. By Proposition~\ref{rec-cnk}, this sequence satisfies
\[
c_n^{(d)}=\sum_{m\geq 1}b_{n,m}^{(d)},
\]
where $b_{n,m}^{(d)}:=b_{n,n,m}^{(d)}$  satisfies
\begin{equation}\label{rec-bnm}
b_{n,m}^{(d)}=\binom{m+nd-2}{d-1}\sum_{j=1}^{m}b_{n-1,j}^{(d)}.
\end{equation}
The recurrence (\ref{rec-bnm}) can be brought in a slightly easier form. 

\begin{lmm}
\label{lem:cn-bnm}
We have,
\begin{align}
\label{eq:recbnm}
b_{n,m}^{(d)}=\frac{dn+m-2}{dn+m-d-1}b_{n,m-1}^{(d)}+\binom{dn+m-2}{d-1}b_{n-1,m}^{(d)},\qquad (n\geq 2,0\leq m\leq n)
\end{align}
with initial conditons
$b_{1,1}^{(d)}=1$ and $b_{n,m}^{(d)}=0$ for (i) $n\geq 2$ and $m=-1$; (ii) $n=1$ and $m=0$; and (iii) $n<m$.
\end{lmm}

\begin{proof}
The recursive structure in (\ref{rec-bnm}) yields
\[
\frac{b_{n,m}^{(d)}}{\binom{dn+m-2}{d-1}}-\frac{ b_{n,m-1}^{(d)}}{\binom{dn+m-3}{d-1}}=b_{n-1,m}^{(d)}.
\]
This gives the  claimed recurrence and the initial conditions are easily checked.
\end{proof}

\newcommand{\dd}{e^{(d)}}
\newcommand{\ddh}{\hat{e}^{(d)}}

To this recurrence, we apply now the method from \cite{ElFaWa}.
Due to the similarities, we only discuss the main differences.
We start with the following transformation of $(b_{n,m}^{(d)})_{0 \leq m \leq n}$ to $(\dd_{i,j})_{\substack{0 \leq i \leq j \\ i-j \text{ even}}}$, which changes the indices and captures the exponential and superexponential terms coming from the binomial coefficient in~\eqref{eq:recbnm}.

\begin{lmm}\label{mod-rec}
We have
\begin{align*}
     b^{(d)}_{n,m} &= \lambda(d)^n (n!)^{d-1} \dd_{n+m,n-m}
     \qquad \text{ with } \qquad
     \lambda(d) = \frac{(d+1)^{d-1}}{(d-1)!},
\end{align*}
where $\dd_{n,m}$ satisfies the following recurrence
{
\small
\begin{align}
    \label{eq:recdnm}
    \dd_{n,m} &= \mu_{n,m}^{(d)} \dd_{n-1,m+1} + \nu_{n,m}^{(d)} \dd_{n-1,m-1}
\end{align}
}
with
\[
\mu^{(d)}_{n,m}=1+\frac{2(d-1)}{(d+1)n + (d-1)m - 2(d+1)}\quad\text{and}\quad \nu^{(d)}_{n,m}=\prod_{i=2}^d \left(1 - \frac{2(m+i)}{(d+1)(n+m)}\right)
\]
for $n \geq 3$ and $m \geq 0$, where $\dd_{n,-1}=\dd_{2,n}=0$ except for $\dd_{2,0}=1/\lambda(d)$.
\end{lmm}

Now, we are interested in
\[
\dd_{2n,0} = \frac{b^{(d)}_{n,n}}{\lambda(d)^n (n!)^{d-1}}
\]
because by the previous lemmas and \eqref{rec-bnm} we have
{
\newcommand{\myeq}{\hspace{-0.5pt}=\hspace{-0.5pt}}
\begin{align}
    \label{eq:thetachain}
    \mathrm{TC}^{(d)}_n
    \myeq \Theta\left( \mathrm{TC}^{(d)}_{n,n-1} \right)
    \myeq \Theta\left( n!c_{n-1}^{(d)} \right)
    \myeq \Theta\left( n! \,  n^{1-d}  b_{n,n}^{(d)} \right)
    \myeq \Theta\left( (n!)^d  \lambda(d)^n n^{1-d}  \dd_{2n,0} \right)%
    \!.
\end{align}
}%
Moreover, observe that for the Theta-result, the initial value of $\dd_{2,0}$ is irrelevant, as it creates only a constant factor. So we may set it to $\dd_{2,0}=1$, or any convenient constant.
Note that this recurrence is very similar to that of relaxed trees~\cite[Equation~(2)]{ElFaWa}, yet with more complicated factors. Observe also that this is exactly recurrence~\cite[Equation~(10)]{FuYuZh} for $d=2$.

\newcommand{\bb}{B}
\newcommand{\aiarg}{\alpha}

Motivated by experiments for large $n$, we use the following ansatz
\begin{align*}
    \dd_{n,m} \approx h(n) f\left(\frac{m+1}{n^{1/3}}\right),
\end{align*}
where $h$ and $f$ are some ``regular'' functions. Next, we substitute $s(n)=h(n)/h(n-1)$ and $m=\kappa n^{1/3}-1$ into~\eqref{eq:recdnm}. Then, for $n\to\infty$ we get the expansion
\begin{align*}
    f(\kappa)s(n) &= 2 f(\kappa) + \left(f''(\kappa) - \frac{2(d-1)}{d+1} \kappa f(\kappa)\right) n^{-2/3} + {\mathcal O}\left(n^{-1}\right).
\end{align*}
Hence, we may assume that
\[
s(n) = 2 + c_1 n^{-2/3} + c_2 n^{-1} + \dots
\]
and this implies that $f(\kappa)$ satisfies the differential equation
\[
    f''(\kappa) = \left(c_1 + \frac{2(d-1)}{d+1}\kappa\right)f(\kappa)
\]
that is solved by the Airy function $\Ai$ of the first kind, as we have $e_{n,m}=0$ for $m>n$ which corresponds to $\lim_{x \to \infty} f(x)=0$.
Additionally, the boundary conditions allow us to compute $c_1$ and we get that
\begin{align}
    \label{eq:fansatz}
    f(\kappa) = C \Ai \left( a_1 + \bb^{1/3} \kappa \right),
    \qquad
    \text{ where }
    \qquad
    \bb: = \frac{2(d-1)}{d+1},
\end{align}
$a_1 \approx 2.338$ is the largest root of the Airy function $\Ai$, and $C$ is an arbitrary constant.
From this we get that $c_1 = a_1 \bb^{1/3}$.
These heuristic arguments guide us to the following results.
The proofs are analogous to~\cite{ElFaWa,ElFaWa2020,FuYuZh}; for the details we refer to the accompanying Maple worksheet~\cite{Wa}.
Note that the next two results generalize~\cite[Propositions~4 and 5]{FuYuZh}, whose results are recovered by setting $d=2$.
Its technical proofs can be found in Appendix~\ref{App-C}.

\begin{pro}
	\label{lem:AiryXLower}
	For all $n,m\geq0$ let
	\begin{align*}	
		\tilde{X}_{n,m} &:= \left(1-\frac{2d-1}{3(d+1)}\frac{m^2}{n} - \frac{3d^2+12d-11}{6(d+1)}\frac{m}{n}\right){\rm Ai}\left(a_{1}+\frac{\bb^{1/3}(m+1)}{n^{1/3}}\right)\quad\text{and}\\
		 \tilde{s}_n &:= 2+\frac{a_1 \bb^{2/3}}{n^{2/3}} - \frac{3d^2-5d+4}{3(d+1)n} - \frac{1}{n^{7/6}}.
	\end{align*}
	Then, for any $\varepsilon>0$, there exists an $\tilde{n}_0$ such that
	\begin{align*}
		\tilde{X}_{n,m}\tilde{s}_{n} \leq \mu_{n,m}^{(d)} \tilde{X}_{n-1,m+1} \!+\! \nu_{n,m}^{(d)} \tilde{X}_{n-1,m-1}
	\end{align*}
	for all $n\geq \tilde{n}_0$ and for all $0 \leq m < n^{2/3-\varepsilon}$, where $\mu^{(d)}_{n,m}$ and $\nu^{(d)}_{n,m}$ are as in Lemma~\ref{mod-rec}.
\end{pro}

\begin{pro}
	\label{lem:AiryXUpper}
	Choose $\eta > \frac{(2d-1)^2}{18(d+1)^2}$ fixed and for all $n,m \geq 0$ let
	\begin{align*}
		\hat{X}_{n,m} &:= \left(1 - \frac{2d-1}{3(d+1)}\frac{m^2}{n} - \frac{3d^2+12d-11}{6(d+1)}\frac{m}{n} + \eta\frac{m^4}{n^2}\right){\rm Ai}\left(a_{1}+\frac{\bb^{1/3}(m+1)}{n^{1/3}}\right)\quad\text{and}\\
		\hat{s}_n &:= 2 + \frac{a_1 \bb^{2/3}}{n^{2/3}} -  \frac{3d^2-5d+4}{3(d+1)n} + \frac{1}{n^{7/6}}.
	\end{align*}
	Then, for any $\varepsilon>0$, there exists a constant $\hat{n}_0$ such that
	\begin{align*}
		\hat{X}_{n,m}\hat{s}_{n} \geq \mu_{n,m}^{(d)} \hat{X}_{n-1,m+1} \!+\! \nu_{n,m}^{(d)} \hat{X}_{n-1,m-1}
	\end{align*}
	for all $n\geq \hat{n}_0$ and all $0 \leq m < n^{1-\varepsilon}$.
\end{pro}

These Propositions allow us now to prove Theorem~\ref{gen-tc-max-k} on the asymptotics of general $d$-combining tree-child networks with $n$ leaves.

\begin{proof}[Proof of Theorem~\ref{gen-tc-max-k}]
    Let us start with the lower bound.

    We first define a sequence $X_{n,m} : =\max\{\tilde{X}_{n,m},0\}$ which satisfies the inequality of Proposition~\ref{lem:AiryXLower} for \emph{all} $m \leq n$.
    Then, we define an explicit sequence $\tilde{h}_n := \tilde{s}_n \tilde{h}_{n-1}$ for $n >0$ and $\tilde{h}_0 = \tilde{s}_0$.
    From this, we get by induction that $\dd_{n,m} \geq C_0 \tilde{h}_{n} X_{n,m}$ for some constant $C_0>0$ and all $n \geq \tilde{n}_0$ and all $0 \leq m \leq n$. Hence,
   \begin{align*}
       \dd_{2n,0}
       &\geq C_0 \tilde{h}_{2n} X_{2n,0}\\
       &\geq C_0 \prod_{i=1}^{2n} \left(2+\frac{a_1 \bb^{2/3}}{i^{2/3}} - \frac{3d^2-5d+4}{3(d+1)i} - \frac{1}{i^{7/6}} \right){\rm Ai}\left(a_{1}+\frac{\bb^{1/3}}{(2n)^{1/3}}\right) \\
       &\geq C_1 4^n e^{3 a_1 (\bb/2)^{2/3} n^{1/3}} n^{\frac{d^2+d-2}{2(d+1)}}.
   \end{align*}
   Finally, combining this with~\eqref{eq:thetachain} we get the lower bound.

   The upper bound is similar, yet more technical.

   The starting point is Proposition~\ref{lem:AiryXUpper} and a function $X_{n,m}$ that is valid for \emph{all} $0 \leq m \leq n$.
   For this purpose we define a sequence $\ddh_{n,m}$ such that $\ddh_{n,m} := \dd_{n,m}$ for $0 \leq m \leq n^{1-\varepsilon}$ and $\ddh_{n,m} := 0$ otherwise; compare with~\cite{ElFaWa,ElFaWa2020}.
   The missing key step is now to show that $\dd_{2n,0} ={\mathcal O}(\ddh_{2n,0})$.
   Combining this with the analogous computations performed for the lower bound we get
   \begin{align*}
       \ddh_{2n,0}
       &\leq \hat{C}_1 4^n e^{3 a_1 (\bb/2)^{2/3} n^{1/3}} n^{\frac{d^2+d-2}{2(d+1)}}.
   \end{align*}

   To complete the prove we show $\dd_{2n,0} \leq 2 \ddh_{2n,0}$ using lattice path theory and computer algebra.
   The argument follows along the same lines as in~\cite[Appendix]{FuYuZh}.
   We start from Equation~\eqref{eq:recdnm} of $\dd_{n,m}$, which we interpret as a recurrence counting lattice paths.
   They are composed of steps $(1,1)$ weighted by $\mu_{n,m}^{(d)}$ and $(1,-1)$ weighted by $\nu_{n,m}^{(d)}$ when the respective step ends at $(n,m)$.
   The total weight of a path is the product of its weights.
   Now, we are interested in the paths never crossing $y=0$ and ending at $(2n,0)$.
   Let now $p_{\ell,k,2n}$ be the number of such paths starting at $(\ell,k)$ and ending at $(2n,0)$.
   From~\eqref{eq:recdnm} we directly get
   \begin{align*}
       p_{\ell,k,2n} =
        \mu_{\ell+1,k-1}^{(d)} p_{\ell+1,k-1,2n}
        +
        \nu_{\ell+1,k+1}^{(d)} p_{\ell+1,k+1,2n},
   \end{align*}
   with $p_{\ell,-1,2n}=0$ and $p_{2n,k,2n} = \delta_{k,0}$.

   Now, as in~\cite{FuYuZh} we are able to show that
   \begin{align}
        \label{eq:plj2ninequ}
       \frac{p_{\ell,j,2n}}{(j+1)^2} \geq \frac{p_{\ell,k,2n}}{(k+1)^2},
   \end{align}
   for integers $0 \leq j < k  \leq \ell \leq 2n$ such that $2 \mid k-j$.
   For the technical details, using reverse induction on $\ell$, we refer to our accompanying Maple worksheet~\cite{Wa}.

   Finally, from~\eqref{eq:plj2ninequ} we directly get
   \begin{align}
        \label{eq:p-2x2y-2x0}
       p_{2x,2y,2n} \leq (2y+1)^2 p_{2x,0,2n},
   \end{align}
   which we need to apply~\cite[Lemma~4.6]{ElFaWa} together with the bound $\dd_{2x,2y} \leq \binom{2x}{x+y}$, which holds due to the same reasons as in~\cite{FuYuZh}: combining the weights of up an down steps gives a weight less than one.
   This proves $\dd_{2n,0} \leq 2 \ddh_{2n,0}$ and ends the proof of Theorem~\ref{gen-tc-max-k}.
\end{proof}

\begin{rem}
   Note that in~\cite{ElFaWa,ElFaWa2020} a stronger result than~\eqref{eq:plj2ninequ} was proved, where the powers in the denominators are $1$ instead of $2$.
   However, any polynomial coefficient in~\eqref{eq:p-2x2y-2x0} suffices to get the same result using~\cite[Lemma~4.6]{ElFaWa}.

   With the same strategy as in~\cite{ElFaWa,ElFaWa2020}, we could show the stronger result $p_{2x,2y,2n} \leq 2 (2y+1) p_{2x,0,2n}$, however, then other technicalities arise: the value $j=0$ and the range $0 \leq j < k  \leq \ell \leq 2n$ have to be treated separately.
\end{rem}

\subsection{Number of Reticulation Nodes}\label{numb-ret-tc}

This section contains the proof of Theorem~\ref{ll-gen-tc}. Since the proof for $d=2$ and $d\geq 3$ is different, we split it into two subsections (Section~\ref{ll-d-2} and Section~\ref{ll-d-larger-2} below). A final subsection contains the proof of Corollary~\ref{cor-1}.

\subsubsection{Bicombining Networks}\label{ll-d-2}

In this section, we consider the case $d=2$. For convenience, we drop the superindex in the notation.

We start with the following bounds for $\mathrm{TC}_{n,k}$.

\begin{lmm}
For $1\leq k\leq n-1$,
\begin{equation}\label{ub-lb}
\frac{n-k}{k(3n-k-3)}\mathrm{TC}_{n,n-k}\leq\mathrm{TC}_{n,n-1-k}\leq\frac{1}{2k}\mathrm{TC}_{n,n-k}.
\end{equation}
\end{lmm}
\pf The upper bound follows from (\ref{tc-number-ub}).

For the lower bound, we generalize the argument from \cite[Lemma 3]{FuYuZh}. Therefore, consider a tree-child network $N$ with $n$ leaves and $n-1-k$ reticulation nodes. From Lemma~\ref{free-nodes}, we know that $N$ has $2k$ free edges. Moreover, $N$ has $3n-k-2$ edges which do not end in a reticulation node; see (\ref{tree-nodes}). Now, by inserting a node into a tree edge and connecting it to a node which is inserted into a free edge, we obtain at most $2k(3n-k-3)\mathrm{TC}_{n,n-1-k}$ tree-child networks with $n$ leaves and $n-k$ reticulation nodes (as those with cycles have to be discarded). On the other hand, each network is created from a latter network exactly $2(n-k)$ times. Thus,
\[
2(n-k)\mathrm{TC}_{n,n-k}\leq 2k(3n-k-3)\mathrm{TC}_{n,n-1-k}.
\]
which proves the lower bound.\qed

From the above bounds, we deduce the following lemma.

\begin{lmm}\label{ub-lb-2}
We have,
\[
\frac{1}{3^kk!}(1+o(1))\mathrm{TC}_{n,n-1}\leq\mathrm{TC}_{n,n-1-k}\leq\frac{1}{2^kk!}\mathrm{TC}_{n,n-1}
\]
uniformly in $k=o(\sqrt{n})$.
\end{lmm}
\pf The upper bound follows from iterating the upper bound of (\ref{ub-lb}); see also (\ref{TCnk-upper-bound}).

For the lower bound, observe that
\[
\frac{n-k}{k(3n-k-3)}=\frac{1}{3k}\left(1+{\mathcal O}\left(\frac{k}{n}\right)\right).
\]
Thus, by iterating the lower bound in (\ref{ub-lb}):
\[
\frac{1}{3^kk!}\left(1+{\mathcal O}\left(\frac{k}{n}\right)\right)^k\mathrm{TC}_{n,n-1}\leq\mathrm{TC}_{n,n-1-k}.
\]
From this the result follows since for the indicated range of $k$, we have
\[
\left(1+{\mathcal O}\left(\frac{k}{n}\right)\right)^k=1+{\mathcal O}\left(\frac{k^2}{n}\right)=1+o(1).
\]
This concludes the proof of the lemma.\qed

We next denote by $\mathrm{F}_{n,k}$ resp. $\mathrm{NF}_{n,k}$ the number of tree-child networks with $n$ leaves and $k$ reticulation nodes whose child of the root is free resp. not free.

We start with an easy observation.
\begin{lmm}\label{non-free}
For $1\leq k\leq n-1$, we have $(2k) \mathrm{TC}_{n,n-1-k}=\mathrm{NF}_{n,n-k}$.
\end{lmm}
\pf A tree-child network with $n$ leaves and $n-1-k$ reticulation nodes has $2k$ free edges; see Lemma~\ref{free-nodes}. Taking any of these free edges and the root edge, inserting nodes in both edges and connecting the node inserted into the root edge with the other node gives a tree-child network with $n$ leaves and $n-k$ reticulation nodes that is not free. Moreover, this construction is clearly reversible.\qed

\begin{rem}
Note that $\mathrm{NF}_{n,n-1}=\mathrm{TC}_{n,n-1}$. Thus, for $k=1$, the above result shows that equality in~\eqref{tc-number-ub} holds for $k=n-2$; compare with Remark~\ref{eq-max-ret}.
\end{rem}

Another easy observation is the following.
\begin{lmm}\label{free-1}
For $1\leq k\leq n-1$,
\[
\mathrm{F}_{n,n-1-k}\leq\frac{1}{2^{k-1}(k-1)!}\mathrm{F}_{n,n-2}.
\]
\end{lmm}
\pf A tree-child network with $n$ leaves and $n-1-k$ reticulation nodes whose child of the root is free has $2k$ free edges of which $2k-2$ are not the edges from the child of the root to its children. By picking one of the latter two edges, one of the remaining $2k-2$ edges, inserting nodes and connecting the former edge to the latter, we obtain a tree-child network with $n$ leaves and $n-k$ reticulation nodes whose child of the root is again free. Conversely, every such network is obtained by this construction at most $2$ times. Thus,
\[
2(2k-2)\mathrm{F}_{n,n-1-k}\leq 2\mathrm{F}_{n,n-k}
\]
or
\[
\mathrm{F}_{n,n-1-k}\leq\frac{1}{2(k-1)}\mathrm{F}_{n,n-k}.
\]
Iterating this gives the claimed result.\qed

The final result we need is the following.
\begin{lmm}\label{free-2}
We have,
\[
\mathrm{F}_{n,n-2}=\mathcal{O}\left(\frac{\mathrm{TC}_{n,n-1}}{n^{2/3}}\right).
\]
\end{lmm}
\begin{proof}
Let $N$ be a network with $n$ leaves and $n-2$ reticulation nodes whose child of the root is free. Note that the two words constructed from $N$ in the proof of Theorem~\ref{formula-TC} both start with $a$ and that this is the sole letter which occurs only twice. Conversely, all words with this property arise from networks with $n$ leaves and $n-2$ reticulation nodes whose child of the root is free. Thus, with the same arguments as in the proof of Theorem~\ref{formula-TC}, we have
\[
\mathrm{F}_{n,n-2}=\frac{n!}{2}g_{n-1},
\]
where $g_{n-1}$ is the number of words in $\mathcal{C}_{n-1,n-2}$ which start with $a$ and this is the sole letter which occurs twice. Next, with the same arguments as used in the proof of Proposition~\ref{rec-cnk}:
\[
g_n=\sum_{m\geq 1}h_{n,m}, \qquad\text{where}\qquad h_{n,m}=(n+m+n-4)\sum_{j=1}^{m}h_{n-1,j}.
\]
We now apply to this sequence the same method as in the last section, where we only need an upper bound. This gives
\[
g_n=\mathcal{O}\left(n!12^n e^{a_1(3n)^{1/3}}n^{-4/3}\right)
\]
and thus,
\[
\mathrm{F}_{n,n-2}=\mathcal{O}\left((n!)^2 12^n e^{a_1(3n)^{1/3}}n^{-7/3}\right).
\]
Comparing with the Theta-result for $\mathrm{TC}_{n,n-1}$ from Theorem~\ref{gen-tc-max-k} (which was also the main result of \cite{FuYuZh}) gives the claimed result.
\end{proof}

Now, we can prove the following proposition.

\begin{pro}
We have,
\[
\mathrm{TC}_{n,n-1-k}=\frac{1}{2^kk!}(1+o(1))\mathrm{TC}_{n,n-1}
\]
uniformly for $0\leq k\leq c\log n$ where $c=1/(3\log(3/2))$.
\end{pro}

\pf First note that
\begin{equation}\label{F-and-NF}
\mathrm{TC}_{n,n-k}=\mathrm{NF}_{n,n-k}+\mathrm{F}_{n,n-k}=(2k) \mathrm{TC}_{n,n-1-k}+\mathrm{F}_{n,n-k},
\end{equation}
where we used Lemma~\ref{non-free}.

Next, by using Lemma~\ref{free-1}, Lemma~\ref{free-2} and the lower bound in Lemma~\ref{ub-lb-2}, we obtain that
\[
\mathrm{F}_{n,n-k}={\mathcal O}\left(\frac{\mathrm{F}_{n,n-2}}{2^k(k-2)!}\right)={\mathcal O}\left(\frac{\mathrm{TC}_{n,n-1}}{2^k(k-2)!n^{2/3}}\right)={\mathcal O}\left(\left(\frac{3}{2}\right)^k\frac{k}{n^{2/3}}\times\mathrm{TC}_{n,n-k}\right)
\]
for $1\leq k\leq n^{1/4}$ (which is within the range of applicability of Lemma~\ref{ub-lb-2}). Thus, for $1\leq k\leq c\log n$,
\[
\mathrm{F}_{n,n-k}={\mathcal O}\left(\frac{\log n}{n^{1/3}}\times\mathrm{TC}_{n,n-k}\right).
\]
Plugging this into (\ref{F-and-NF}), we obtain
\[
\mathrm{TC}_{n,n-1-k}=\frac{1}{2k}\left(1+{\mathcal O}\left(\frac{\log n}{n^{1/3}}\right)\right)\mathrm{TC}_{n,n-k}
\]
and by iteration
\[
\mathrm{TC}_{n,n-1-k}=\frac{1}{2^kk!}\left(1+{\mathcal O}\left(\frac{\log n}{n^{1/3}}\right)\right)^k\mathrm{TC}_{n,n-1}
\]
from which the result follows since
\[
\left(1+{\mathcal O}\left(\frac{\log n}{n^{1/3}}\right)\right)^k=1+{\mathcal O}\left(\frac{\log^2 n}{n^{1/3}}\right)=1+o(1).
\]
This completes the proof.\qed

\begin{rem}
Note that the last result improves the bounds of Lemma~\ref{ub-lb-2}, but for a smaller range of~$k$. (The value of $c$ in the proposition is not best possible; however, it is sufficient for our purpose.)
\end{rem}

\begin{proof}[Proof of Theorem~\ref{ll-gen-tc}]
We first consider the number of tree-child networks with $n$ leaves. Recall that
\[
\mathrm{TC}_{n}=\sum_{k=0}^{n-1}\mathrm{TC}_{n,n-1-k}.
\]
Let $k^{*}=c\log n$ with $c$ from the last proposition. Then,
\[
\mathrm{TC}_{n}=\sum_{k\leq k^{*}}\mathrm{TC}_{n,n-1-k}+\sum_{k^{*}<k\leq n-1}\mathrm{TC}_{n,n-1-k}.
\]
Using the upper bound in (\ref{ub-lb-2}), we obtain
\[
\sum_{k^{*}<k\leq n-1}\mathrm{TC}_{n,n-1-k}\leq\mathrm{TC}_{n,n-1}\sum_{k>k^{*}}\frac{1}{2^kk!}=o(\mathrm{TC}_{n,n-1}).
\]
On the other hand, by the last proposition:
\begin{align*}
\sum_{k\leq k^{*}}\mathrm{TC}_{n,n-1-k}&=(1+o(1))\mathrm{TC}_{n,n-1}\sum_{k\leq k^{*}}\frac{1}{2^kk!}\\ &=(1+o(1))\mathrm{TC}_{n,n-1}\sum_{k=0}^{\infty}\frac{1}{2^kk!}\\
&=(1+o(1))\mathrm{TC}_{n,n-1}e^{1/2}.
\end{align*}
Combining the last two displays gives
\begin{equation}\label{asymp-TCn-TCnn-1}
\mathrm{TC}_{n}\sim e^{1/2}\mathrm{TC}_{n,n-1},\qquad (n\rightarrow\infty).
\end{equation}
Thus, for all fixed $k$:
\[
{\mathbb P}(n-1-R_n=k)=\frac{\mathrm{TC}_{n,n-1-k}}{\mathrm{TC}_{n}}\longrightarrow\frac{e^{-1/2}}{2^kk!},\qquad (n\rightarrow\infty)
\]
which proves the claimed Poisson limit law.
\end{proof}

\begin{rem}\label{d=2-conv-mom}
With the same arguments as used to prove (\ref{asymp-TCn-TCnn-1}), we can also show that
\[
\sum_{k=0}^{n-1}k\mathrm{TC}_{n-1-k}\sim \frac{1}{2}e^{1/2}\mathrm{TC}_{n,n-1}
\]
and thus ${\mathbb E}(n-1-R_n)\sim 1/2$. Moreover, in a similar way, higher moments of $n-1-R_n$ can be shown as well to converge to those of $\mathrm{Poisson}(1/2)$.
\end{rem}

\subsubsection{\texorpdfstring{$d$-combining Networks with $d\geq 3$}{d-combining Networks with d>=3}}\label{ll-d-larger-2}

Here, we prove Theorem~\ref{ll-gen-tc} for $d\geq 3$. For the sake of simplicity, we restrict ourselves to the case $d=3$; the case of larger $d$ follows along similar lines.%

Recall the class of words $\mathcal{C}_n^{(d)}$ from Definition~\ref{def-words}. In addition, we denote by  $\mathcal{TC}_{n,k}^{(d)}$ the set of tree-child networks with $n$ leaves and $k$ reticulation nodes. Note that in Proposition~\ref{formula-TC}, we constructed a bijection $f$ from

\[
f:\ \underbrace{\{0,1\}^{n-k-1}\times\mathcal{TC}_{n,k}^{(d)}}_{\displaystyle{=:2^{n-k-1}\times\mathcal{TC}_{n,k}^{(d)}}}\longmapsto\mathcal{C}_{n-1,k}^{(d)}\times\mathcal{S}_n,
\]

\noindent where $\mathcal{S}_n$ denotes the symmetric group of order $n$.

In particular, this bijection implies that $\mathcal{TC}_{n,n-1}^{(d)}$ is in bijection with $\mathcal{C}^{(d)}_{n-1,n-1}\times\mathcal{S}_n$ for $d=2,3$ ($k=n-1$) and $2\times\mathcal{TC}_{n,n-2}^{(d)}$ is in bijection with $\mathcal{C}^{(d)}_{n-1,n-2}\times\mathcal{S}_n$ for $d=2,3$ ($k=n-2$); see the upper half and lower half of Figure~\ref{process8}. Also, from Remark~\ref{eq-max-ret}, we have a bijection from $\mathcal{TC}_{n,n-1}^{(2)}$ to $2\times\mathcal{TC}_{n,n-2}^{(2)}$ (see the middle arrow in the left half of Figure~\ref{process8}) which gives a $2$-to-$1$ map from $\mathcal{TC}_{n,n-1}^{(2)}$ to $\mathcal{TC}_{n,n-2}^{(2)}$ (take a maximally reticulated network and remove the reticulation node associated with the non-free edge of the child of the root, the non-free edge and its initial node; see Remark~\ref{eq-max-ret}) followed by a $1$-to-$2$ map from $\mathcal{TC}_{n,n-2}^{(2)}$ to $2\times\mathcal{TC}_{n,n-2}^{(2)}$ (by picking the newly created free edge which contained the reticulation node in the previous construction).

Using this, we can now prove the following result.

\begin{figure}[!t]
    \centering
    \includegraphics[scale=1]{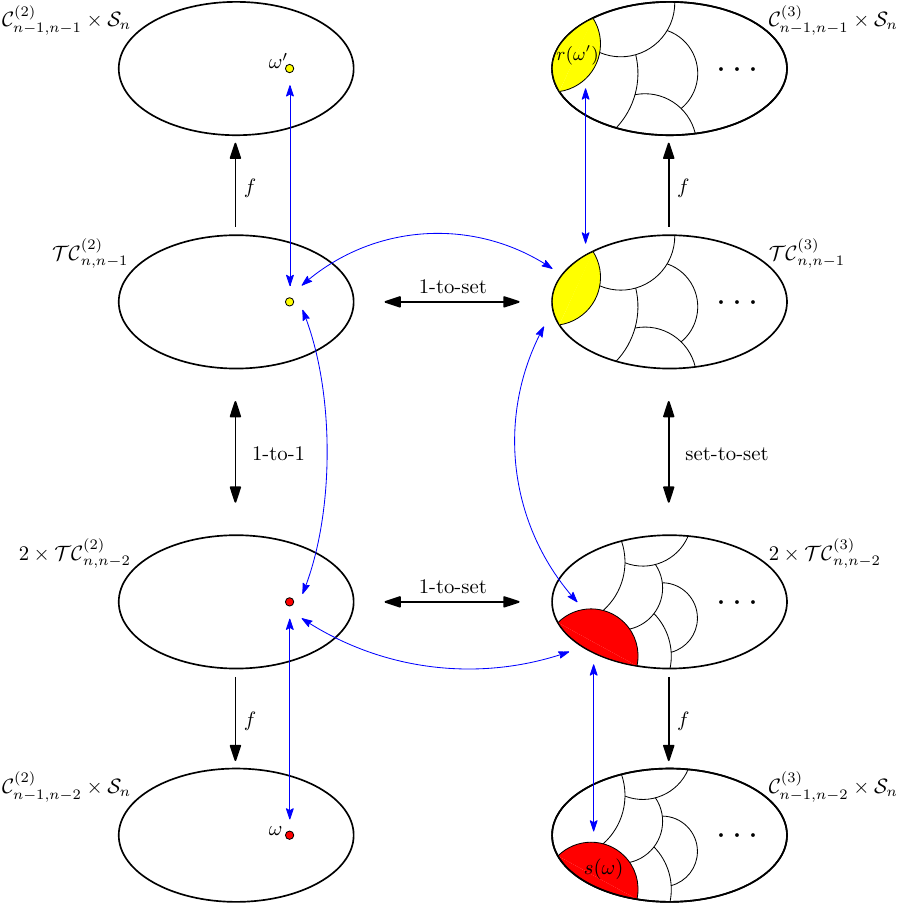}
    \caption{The constructions and maps which are used in the proof of Lemma~\ref{main-lemma}.}\label{process8}
\end{figure}

\begin{lmm}\label{main-lemma}
We have,
\[
\mathrm{TC}_{n,n-2}^{(3)}=o(\mathrm{TC}_{n,n-1}^{(3)}).
\]
\end{lmm}

\begin{proof}
We first explain a construction of $\mathcal{TC}_{n,n-1}^{(3)}$ from $\mathcal{TC}_{n,n-1}^{(2)}$:  for a network from $\mathcal{TC}_{n,n-1}^{(2)}$, we add a new parent (and corresponding edge) for each reticulation node to an edge on the path-components we pass before we read the first parent of the reticulation node in the construction of the word and permutation from the proof of Theorem~\ref{formula-TC}, where the reticulation nodes are processed consecutively (in any order).  Depending on the choice of the edges, we get several networks in $\mathcal{TC}_{n,n-1}^{(3)}$ which all have essentially the same path-component structure as the network from $\mathcal{TC}_{n,n-1}^{(2)}$ (with respect to the encoding from Section~\ref{enc-words}). Conversely, removing the first parent (and corresponding edge) of each reticulation node of a network in $\mathcal{TC}_{n,n-1}^{(3)}$ gives a network in $\mathcal{TC}_{n,n-1}^{(2)}$. Thus, this gives a bijection between networks in $\mathcal{TC}_{n,n-1}^{(2)}$ and classes of networks from $\mathcal{TC}_{n,n-1}^{(3)}$, where these classes form a partition of $\mathcal{TC}_{n,n-1}^{(3)}$; see the second row in the top half of Figure~\ref{process8}.

Equivalently, when viewing networks as words and permutations (where the permutation however is here irrelevant because it does not change in the construction), any word in $\mathcal{C}_{n-1,n-1}^{(3)}$ can be obtained by adding a new $\omega_{i}$  at any position before the first occurrence of $\omega_{i}$ in a word from $\mathcal{C}_{n-1,n-1}^{(2)}$, where $\omega_i$ runs through all letters.

 For example: $baaabb$ (a word in $\mathcal{C}_{2,2}^{(2)}$) leads to the class $\{bbaaaabb, babaaabb, abbaaabb\}$ (words in $\mathcal{C}_{2,2}^{(3)}$); see Figure~\ref{drawex} for the corresponding networks.

 \begin{figure}[!t]
    \centering
    \includegraphics[scale=1.05]{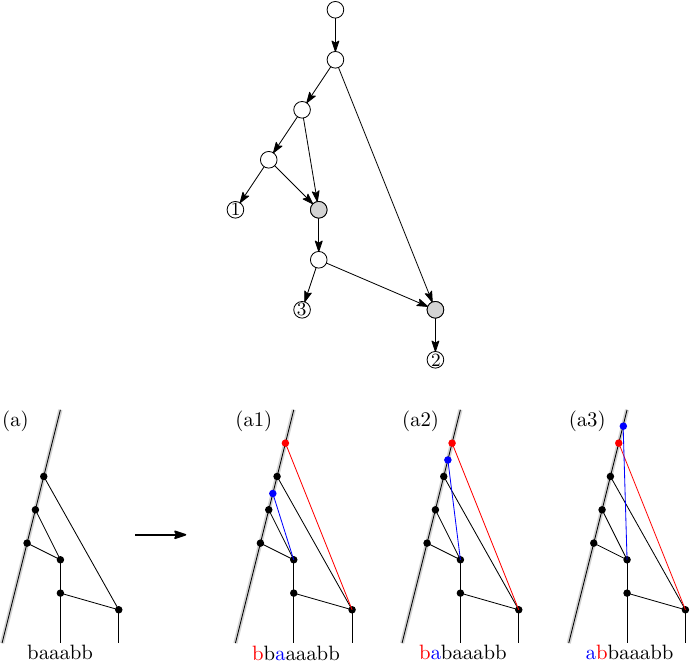}
    \vspace*{0.15cm}
    \caption{The construction of $\mathcal{TC}_{n,n-1}^{(3)}$ from $\mathcal{TC}_{n,n-1}^{(2)}$. Top: a network from $\mathcal{TC}_{n,n-1}^{(2)}$. Bottom: the three networks from $\mathcal{TC}_{n,n-1}^{(3)}$ constructed from the top network; the corresponding words are below each network. (The permutation in each case is $132$; that is why we refrained from indicating it.)}\label{drawex}
\end{figure}

In the next step, we construct $2\times\mathcal{TC}_{n,n-2}^{(3)}$ from $2\times\mathcal{TC}_{n,n-2}^{(2)}$ in a similar way (where we turn networks into maximally reticulated networks as in the proof of Theorem~\ref{formula-TC}). In particular, since every network in $2\times\mathcal{TC}_{n,n-2}^{(2)}$ corresponds to a word from $\mathcal{C}_{n-1,n-2}^{(2)}$ and a permutation (which is however again irrelevant), we apply to the words from $\mathcal{C}_{n-1,n-2}^{(2)}$ the same construction as above with the only difference that we only use the $\omega_i$'s which are repeated $3$ times. Like this, we obtain all words from $\mathcal{C}_{n-1,n-2}^{(3)}$.

For example, take $abbab+132$ and $aabbb+312$ (which encode the same network in $\mathcal{TC}_{3,1}$); $abbab$ leads to the words $\{babbab, abbbab\}$ from $\mathcal{C}_{2,1}^{(3)}$ and $aabbb$ leads to the words $\{baabbbb, ababbb, aabbbb\}$ from $\mathcal{C}_{2,1}^{(2)}$; see Figure~\ref{2 to 3} for a plot of the corresponding networks from $\mathcal{TC}_{3,1}^{(3)}$.

\begin{figure}[!t]
    \centering
    \includegraphics[scale=1.15]{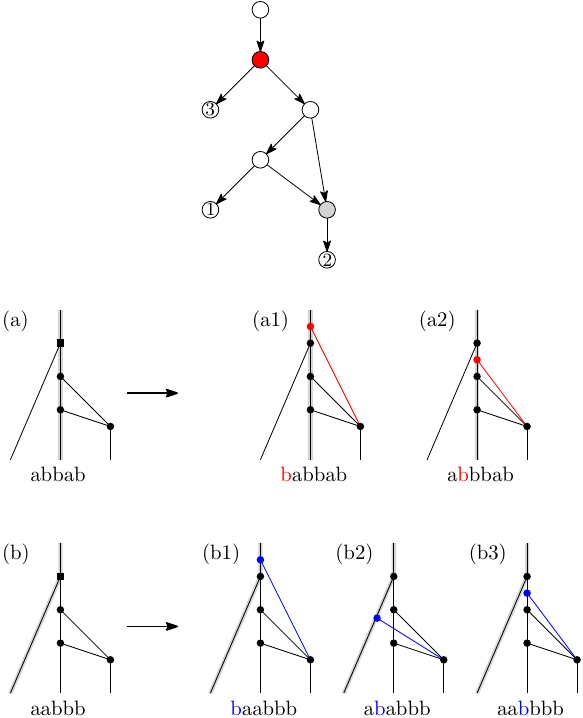}
    \caption{The construction of $\mathcal{TC}_{n,n-2}^{(3)}$ from $\mathcal{TC}_{n,n-2}^{(2)}$. Top: a network from $\mathcal{TC}_{n,n-2}^{(2)}$. Bottom left: the two corresponding networks from $2\times\mathcal{TC}_{n,n-2}^{(2)}$ where the first path-component (using the indexing from the proof of Theorem~\ref{formula-TC}) is in bold. Bottom right: the networks from $\mathcal{TC}_{n,n-2}^{(3)}$ constructed from each of the two networks; the corresponding words are below the networks.}\label{2 to 3}
\end{figure}

Next, as mentioned in the paragraph before the lemma, there is a bijection between $\mathcal{TC}_{n,n-1}^{(2)}$ and $2\times\mathcal{TC}_{n,n-2}^{(2)}$. This bijection gives rise to a bijection between $\mathcal{C}_{n-1,n-1}^{(2)}\times\mathcal{S}_n$ and $\mathcal{C}_{n-1,n-2}^{(2)}\times\mathcal{S}_n$ (which just removes the first letter from a word $\omega'$ of the former to obtain a word $\omega$ of the latter); see left half of Figure~\ref{process8}. (The permutation remains unchanged.) Note that $\omega'$ is bijectively mapped onto a set of words from $\mathcal{C}_{n,n-1}^{(3)}$ and $\omega$ onto a set of words from $\mathcal{C}_{n,n-2}^{(2)}$; see the top and bottom half of Figure~\ref{process8}. Denote the cardinality of these sets by $r(\omega')$ and $s(\omega)$, respectively. Then, to show that $\mathrm{TC}_{n,n-2}^{(3)}=o(\mathrm{TC}_{n,n-1}^{(3)})$, it suffices to show that $s(\omega)=o(r(\omega'))$ uniformly over all $\omega$ in $\mathcal{C}_{n,n-2}^{(2)}$, or equivalently, we have to find a uniform lower bound of the ratio $r(\omega')/s(\omega)$ that tends to infinity. We consider this ratio next.

For example, if $\omega=abbccabc$, then $\omega'=aabbccabc$ and the ratio becomes
\[
\frac{r(aabbccabc)}{s(abbccabc)}=\frac{1\cdot 4\cdot 7}{2\cdot 5}.
\]
More generally,
\[ \frac{r(\omega')}{s(\omega)}=\frac{k_{1}\cdot k_{2}\cdots k_{n}}{(k_{2}-2)\cdots(k_{n}-2)},\]
where $i$ denotes the $i$th  first parent of the letters in $\omega'$ and $k_{i}$ indicates the number of  possibilities of adding an additional parent for $i$ by the above method. Note that $k_1=1$ and the $k_{i}$'s increase, thus, the $k_i/(k_{i}-2)$'s decrease. Moreover, note that for each $k_{i}$, we have $k_{i}\leq4(i-1)+1$ since the upper bound is the extremal case, i.e., the case that each of the previous $i-1$ letters occur $4$ times. Consequently,
\begin{align*}
\frac{r(\omega')}{s(\omega)}=\frac{k_{2}\cdots k_{n}}{(k_{2}-2)\cdots(k_{n}-2)}\geq \frac{5\cdot9\cdots(4n-3)}{3\cdot7\ldots(4n-5)}=
\Theta\left(\frac{\Gamma(n+\frac{1}{4})}{ \Gamma(n-\frac{1}{4})}\right)=\Theta(n^{1/2}),
\end{align*}
 where we used Stirling's formula for the gamma function in the last step. This implies that, $s(\omega)=o(r(\omega'))$ which (as explained above) in turn implies that $\mathrm{TC}_{n,n-2}^{(3)}=o(\mathrm{TC}_{n,n-1}^{(3)})$. This is the claimed result.
\end{proof}

We can now prove the second case of Theorem~\ref{ll-gen-tc}.

\begin{proof}[Proof of Theorem~\ref{ll-gen-tc}-(ii)] First, observe that from (\ref{tc-number-ub}) and the previous lemma, we have
\begin{equation}\label{pf-step-1}
\mathrm{TC}_{n,n-1-k}^{(3)}=o(\mathrm{TC}_{n,n-1}^{(3)})
\end{equation}
for all fixed $k\geq 1$.

Next, by iterating (\ref{tc-number-ub}),
\[
\mathrm{TC}_{n,n-1-k}^{(3)}\leq\frac{1}{2^{k-1}k!}\mathrm{TC}_{n,n-2}^{(3)}
\]
for all $1\leq k\leq n-1$. Consequently,
\[
\sum_{k=2}^{n-1}\mathrm{TC}_{n,n-1-k}^{(3)}={\mathcal O}(\mathrm{TC}_{n,n-2}^{(3)})=o(\mathrm{TC}_{n,n-1}^{(3)}).
\]
Thus,
\begin{equation}\label{pf-step-2}
\mathrm{TC}_n^{(3)}=\mathrm{TC}_{n,n-1}^{(3)}+\mathrm{TC}_{n,n-2}^{(3)}+\sum_{k=2}^{n-1}\mathrm{TC}_{n,n-1-k}^{(3)}\sim \mathrm{TC}_{n,n-1}^{(3)}.
\end{equation}

Now, we can prove the claim:
\[
{\mathbb P}(n-1-T_n^{(3)}=k)={\mathbb P}(T_n^{(3)}=n-1-k)=\frac{\mathrm{TC}_{n,n-1-k}^{(3)}}{\mathrm{TC}_n^{(3)}}\longrightarrow\begin{cases} 1,&\text{if}\ k=0;\\ 0,&\text{if}\ k\geq 1,\end{cases}
\]
where the last step follows from (\ref{pf-step-1}) and (\ref{pf-step-2}).
\end{proof}

\begin{rem}\label{d>2-conv-mom}
Similar as in Remark~\ref{d=2-conv-mom}, we can prove that all moments of $n-1-T_n^{(3)}$ converge to $0$.
\end{rem}

\subsubsection{Proof of Corollary~\ref{cor-1}}
\label{sec:proofcor-1}

Note that
\[
{\mathbb E}(W_n^{(d)})=\sum_{\ell\geq 0}{\mathbb P}(W_n^{(d)}>\ell)\leq\sum_{\ell\geq 0}{\mathbb P}(n-1-T_n^{(d)}>\ell)={\mathbb E}(n-1-T_n^{(d)}).
\]
where the inequality follows from the fact that each twig is a free tree node and the number of free tree nodes is given by $n-1-k$ where $k$ is the number of reticulation nodes; see Lemma~\ref{free-nodes}. The result follows now from:
\[
{\mathbb E}(n-1-T_n^{(d)})\longrightarrow\begin{cases}1/2,&\text{if}\ d=2;\\ 0,&\text{if}\ d\geq 3;\end{cases}
\]
see Remark~\ref{d=2-conv-mom} and Remark~\ref{d>2-conv-mom}.

\section{Conclusion and Open Questions}\label{con}

The main purpose of this paper was to propose and investigate the class of {\it $d$-combining tree-child networks} which generalizes the class of bicombining tree-child networks. The latter class is one of the most important classes of phylogenetic networks; it has been widely studied in recent years. One of our major reasons for generalizing it was that a better understanding of its combinatorial properties was gained by placing it into a larger framework. We give a short summary of the results we obtained.

First, for one-component $d$-combining tree-child network, the availability of an easy counting formula (see Section~\ref{exact-count}) made it possible to derive asymptotic counting results and distributional results for the number of reticulation nodes of random one-component $d$-combining tree-child networks by standard methods (see Section~\ref{numb-ret-otc}). Moreover, we also investigated the Sackin index for this class (see Section~\ref{Sackin}) following \cite{CaZh} where the bicombining case was considered. However, whereas we derived convergence of all moments and limit law results for the number of reticulation nodes, we only proved a Theta-result for the mean of the Sackin index. It would be interesting to derive higher moments and prove a limit law result for the Sackin index, too. Moreover, it would be interesting to prove similar results also for other shape parameters such as the number of cherries.

For general $d$-combining tree-child networks, counting them turned out to be more involved. For the bicombining case, in \cite{PoBa} an encoding by words was proposed which led to a counting formula for this class of networks. However, this encoding was just conjectural. In Section~\ref{enc-words}, we proposed a slightly modified encoding which could be rigorously established (and extends from $d=2$ to $d\geq 3$). Our encoding again led to a formula for the number of networks providing a recursive way of computing these numbers for small values of $n,k,d$ (see Appendix~\ref{App-A}). In addition, we used our encoding to establish a Theta-result for the number of networks (see Section~\ref{asymp-count}) and again proved convergence of moments and a limit distribution result for the number of reticulation nodes (see Section~\ref{numb-ret-tc}). From the latter result, we also derived a preliminary result about the distribution of cherries. Proving more detailed stochastic results for this number as well as investigating the Sackin index (and other parameters) for general $d$-combining tree-child networks is still a challenge and we leave these questions open for further investigations. In addition, it would be also interesting to know whether the results from Appendix~\ref{App-B}, where results for fixed $k$ are discussed with a generalization of the method from \cite{CaZh} from the bicombining to the $d$-combining case, can be derived with our encoding from Section~\ref{enc-words}? Moreover, also the question whether the conjecture from \cite{PoBa} can be generalized to $d$-combining networks is open. (We might treat these questions elsewhere.)

A straightforward (and natural) further generalization of our class of $d$-combining networks would be to allow a finite set, say $\{d_1,\ldots,d_m\}$, of in-degrees for reticulation nodes. (Allowing an infinite set would make the counting problem meaningless.) Indeed, many of our exact results seem to generalize to this situation and asymptotic results should be doable as well. However, so far, we cannot see what can be gained from such a further generalization; the results of this paper seem to continue to hold and no new phenomena seem to arise. That is why we refrained from doing this here and unless such a generalization leads to considerable new mathematical challenges and/or different phenomena and/or additional insights about the (most important) bicombining case, this level of generality will only be explored in the PhD thesis of the first author.

\section{Acknowledgments}

We thank the reviewer for many useful suggestions. Also, we acknowledge financial support from: Y.-S.~Chang and M.~Fuchs were supported by National Science and Technology Council (NSTC) under the grant NSTC-111-2115-M-004-002-MY2; M.~Wallner was supported by the Austrian Science Fund (FWF): P~34142; G.-R.~Yu was supported by NSTC under the grant NSTC-110-2115-M-017-003-MY3.

\appendix
\begin{landscape}
\section{Tables}\label{App-A}

\vspace*{0.6cm}
\begin{table}[!htb]
\centering
\begin{tabular}{c|cccccccc} \hline
\multicolumn{1}{c|}{$n\backslash{}k$} & \multicolumn{1}{c}{0} & \multicolumn{1}{c}{1} & \multicolumn{1}{c}{2} & \multicolumn{1}{c}{3} & \multicolumn{1}{c}{4} & \multicolumn{1}{c}{5} & \multicolumn{1}{c}{6} & \multicolumn{1}{c}{7} \\ \hline
2 & 1 & 2 \\
3 & 3 & 21 & 42 \\
4 & 15 & 228 & 1272 & 2544 \\
5 & 105 & 2805 & 30300 & 154500 & 309000 \\
6 & 945 & 39330 & 696600 & 6494400 & 31534200 & 63068400 \\
7 & 10395 & 623385 & 16418430 & 241204950 & 2068516800 & 9737380800 & 19474761600 \\
8 & 135135 & 11055240 & 405755280 & 8609378400 & 113376463200 & 920900131200 & 4242782275200 & 8485564550400 \\
 \hline
\end{tabular}
\caption{\label{tab:TCN_d2}$\mathrm{TC}^{(2)}_{n,k}$ for $2\leq n\leq 8$ and $0 \leq k < n$; see also \cite{CaZh}.}
\end{table}

\vspace*{0.6cm}
\begin{table}[!htb]
\centering
\begin{tabular}{c|ccccccc} \hline
\multicolumn{1}{c|}{$n\backslash{}k$} & \multicolumn{1}{c}{0} & \multicolumn{1}{c}{1} & \multicolumn{1}{c}{2} & \multicolumn{1}{c}{3} & \multicolumn{1}{c}{4} & \multicolumn{1}{c}{5} & \multicolumn{1}{c}{6} \\ \hline
2 & 1 & 2 \\
3 & 3 & 33 & 150 \\
4 & 15 & 492 & 7908 & 55320 \\
5 & 105 & 7725 & 291420 & 6179940 & 57939000 \\
6 & 945 & 132030 & 9603270 & 430105320 & 11292075000 & 132120450000 \\
7 & 10395 & 2471805 & 307525050 & 24586633890 & 1284266876760 & 40079165452200 & 560319972030000 \\ \hline
\end{tabular}
\caption{\label{tab:TCN_d3}$\mathrm{TC}^{(3)}_{n,k}$ for $2\leq n\leq7$ and $0 \leq k < n$.}
\end{table}

\begin{table}[!htb]
\centering
\begin{tabular}{c|cccccc} \hline
\multicolumn{1}{c|}{$n\backslash{}k$} & \multicolumn{1}{c}{0} & \multicolumn{1}{c}{1} & \multicolumn{1}{c}{2} & \multicolumn{1}{c}{3} & \multicolumn{1}{c}{4} & \multicolumn{1}{c}{5} \\ \hline
2 & 1 & 2 \\
3 & 3 & 48 & 546 \\
4 & 15 & 942 & 45132 & 1243704 \\
5 & 105 & 18375 & 2394360 & 227116260 & 11351644920 \\
6 & 945 & 375705 & 107314200 & 23919407460 & 3724353682560 & 291451508298720 \\ \hline
\end{tabular}
\caption{\label{tab:TCN_d4}$\mathrm{TC}^{(4)}_{n,k}$ for $2\leq n\leq 6$ and $0 \leq k < n$.}
\end{table}

\begin{table}[!htb]
\centering
\begin{tabular}{c|ccccc} \hline
\multicolumn{1}{c|}{$n\backslash{}k$} & \multicolumn{1}{c}{0} & \multicolumn{1}{c}{1} & \multicolumn{1}{c}{2} & \multicolumn{1}{c}{3} & \multicolumn{1}{c}{4} \\ \hline
2 & 1 & 2 \\
3 & 3 & 66 & 2016 \\
4 & 15 & 1650 & 242496 & 28710864 \\
5 & 105 & 39135 & 17566470 & 7876446840 & 2307919133520 \\ \hline
\end{tabular}
\caption{\label{tab:TCN_d5}$\mathrm{TC}^{(5)}_{n,k}$ for $2\leq n\leq 5$ and $0 \leq k < n$.}
\end{table}

\begin{table}[!htb]
\centering
\begin{tabular}{c|ccccc} \hline
\multicolumn{1}{c|}{$n\backslash{}k$} & \multicolumn{1}{c}{0} & \multicolumn{1}{c}{1} & \multicolumn{1}{c}{2} & \multicolumn{1}{c}{3} & \multicolumn{1}{c}{4} \\ \hline
2 & 1 & 2 \\
3 & 3 & 87 & 7524 \\
4 & 15 & 2700 & 1246740 & 676431360 \\
5 & 105 & 76515 & 118491090 & 262058953860 & 483098464854720 \\ \hline
\end{tabular}
\caption{\label{tab:TCN_d6}$\mathrm{TC}^{(6)}_{n,k}$ for $2\leq n\leq 5$ and  $0 \leq k < n$.}
\end{table}
\end{landscape}

\section{The Method of Component Graphs}\label{App-B}

The purpose of this appendix is to explain that the method of component graphs from \cite{CaZh} also extends from the bicombining case to the $d$-combining case. The method yields another formula for $\mathrm{TC}_{n,k}^{(d)}$ (see Section~\ref{ano-form-TC} below) which in the bicombining case was used in \cite{CaZh} to (a) compute values for small $n$ and $k$ and (b) derive formulas for $k=1$ and $k=2$. The same can be done with our extension in the $d$-combining case (see Section~\ref{formulas-fixed-k} below), however, the computation of values is more effective with the method proposed in Section~\ref{formula-TC} (even in the bicombining case). Moreover, the method from \cite{CaZh} was used in \cite{FuHuYu} to obtain the first-order asymptotics of $\mathrm{TC}_{n,k}^{(2)}$ for fixed $k$ as $n$ tends to infinity; again this carries over to the $d$-combining case (see Section~\ref{asymp-fixed-k} below).

\subsection{Exact Counting}\label{ano-form-TC}

The main idea of \cite{CaZh} was to reduce each tree-child network to its component graph and then conversely build all tree-child networks from their component graphs. Before explaining this in detail, we give the definition of a component graph.

\begin{df}[Component graph]
A component graph is a (rooted) vertex-labeled DAG such that the following properties hold.
\begin{enumerate}
\item[(i)] Every non-root node has in-degree equal to $d$;
\item[(ii)] the root has in-degree equal to $0$;
\item[(iii)] multi-edges between two nodes are allowed.
\end{enumerate}
\end{df}

We denote by $\mathcal{K}_m^{(d)}$ the set of component graphs with $m$ nodes and by $\mathcal{K}_{m,s}^{(d)}$ the set of component graphs with $m$ nodes of which $s$ are leaves. Set $k_{m}^{(d)} = \vert \mathcal{K}_m^{(d)}\vert$ and $k_{m,s}^{(d)}=\vert\mathcal{K}_{m, s}^{(d)}\vert$. Then,
\[
k_{m}^{(d)}=\sum_{s=1}^{m-1}k_{m,s}^{(d)},
\]
since the number of leaves of a component graph is at least $1$ and at most $m-1$ (for the star-component graph; see Figure~\ref{star-component}).

This, together with the following recursive formula for $k_{m,s}^{(d)}$, makes it possible to compute $k_m^{(d)}$. (See \cite[Theorem~15]{CaZh} for the bicombining case.)

\begin{thm}\label{numb-comp-graphs} For $m\geq 2$,
\[
k_{m,s}^{(d)} = \sum_{1\leq t\leq m-s-1}\, \binom{m}{s} \beta^{(d)}(m,s,t)\, k_{m-s,t}^{(d)},\qquad (1\leq s\leq m-1),
\]
with initial condition $k_{1,1}^{(d)}=1$ and
\[
\beta^{(d)}(m,s,t) = \sum_{0\leq\ell\leq t} (-1)^{\ell}\binom{t}{\ell}\binom{m-s-\ell+d-1}{d}.
\]
\end{thm}

\begin{proof}
The recurrence can be obtained by the following way of constructing all component graphs in $\mathcal{K}_{m,s}^{(d)}$ from those in $\mathcal{K}_{m-s,t}^{(d)}$:
\begin{enumerate}
\item[(i)] Choose $t$ with $1\leq t\leq m-s-1$ and a graph $G$ in $\mathcal{K}_{m-s,t}^{(d)}$;
\item[(ii)] Add $s$ new nodes, labelled by $\{1',\cdots,s'\}$, to $G$ such that (a) these nodes become the new leaves and (b)  all old leaves have at least one out-going edge (i.e., all of them become internal nodes). By the inclusion–exclusion principle, there are
\[
\beta^{(d)}(m,s,t)=\sum_{0 \leq\ell\leq t} (-1)^{\ell}\binom{t}{\ell} \binom{m-s-\ell+d-1}{d}^s
\]
ways of choosing the $d$ incoming edges for the new leaves; here, the counting is done such that $\ell$ of the old leaves are not used as parents of the new leaves;
\item[(iii)] Choose $s$ leaves from the set of labels $\{1,\ldots,m\}$ and use them to re-label the new leaves; use the remaining labels to re-label the remaining nodes in an order-consistent way.\qedhere
\end{enumerate}
\end{proof}

\begin{figure}[t]
\centering
\includegraphics[scale=1]{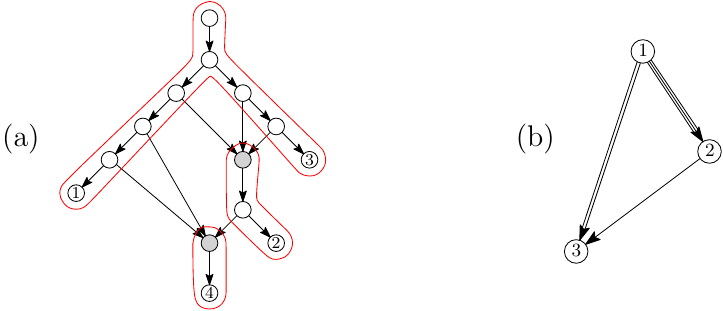}
\caption{(a) The $3$-combining network (with two reticulation nodes) from Figure~\ref{pn-fig}, (a). The three tree-components when incoming edges of the reticulation nodes are removed are encircled in red. (b) The corresponding component graph.}\label{3-comg-ex}
\end{figure}

Now assume that a tree-child networks $N$ with $k$ reticulation nodes is given. Then, its component graph is obtained as follows. First, remove all the incoming edges of the reticulation nodes. The resulting graph is a forest consisting of $k+1$ (directed) trees (called {\it tree-components}) which are either rooted at the network root or at a reticulation node; see Figure~\ref{3-comg-ex}, (a). The component graph of $N$ is then the DAG whose vertex set corresponds to the set of tree-components with an edge between vertices if there was an edge (which was deleted in the first step above) connecting the two tree-components. The vertices are labeled as follows: consider the set of leaf-labels of each tree-component which partition $\{1,\ldots,n\}$. The blocks of this partition can be indexed by the rank of the smallest element of the block; these indices are then given as label to the corresponding nodes of the component graph; see Figure~\ref{3-comg-ex}, (b) for an example.

The above construction reduces every tree-child network with $k$ reticulation nodes to a component graph with $k+1$ nodes. Conversely, all tree-child networks with $k$ reticulation nodes can be obtained from the component graphs with $k+1$ nodes by a blow-up procedure: first choose a component graph with $k+1$ nodes and a set partition of $\{1,\ldots,n\}$ into $k+1$ blocks. Order the blocks according to the ranks of their smallest element and distribute them to the vertices of the (chosen) component graph so that the node with label $j$ receives the $j$th block ($1\leq j\leq k+1$). For each node, choose a phylogenetic tree whose set of leaf-labels corresponds to the block of the partition which was assigned to the node. Next, add nodes on the edges of the phylogenetic trees such that the number of added nodes equals the out-degrees of the nodes of the component graph. Finally, connect these nodes to the roots of the phylogenetic trees in the same way as the corresponding nodes in the component graph are connected. The resulting networks are all tree-child networks with $k$ reticulation nodes.

The blow-up procedure just described immediately translates into a formula for $\mathrm{TC}_{n,k}^{(d)}$. (See \cite[Theorem~16]{CaZh} for the bicombining case.)

\begin{thm}
Let $\prod_{n,k+1}$ be the set of partitions of $\{1,\ldots,n\}$ into $k+1$ blocks. Then,
\begin{equation}\label{formula-TC-2}
\mathrm{TC}_{n,k}^{(d)} =\frac{1}{2^{n-k-1}}\,\sum_{\{B_j\}_{j=1}^{k+1} \in \prod_{n,k+1}}\,\sum_{G\in \mathcal{K}_{k+1}^{(d)}} \,\prod_{j=1}^{k+1}\,\frac{(2b_j+g_j-2)!}{(b_j-1)! \prod_{\ell=1}^{k+1}(g_{j,\ell})!},
\end{equation}
where $b_j=\vert B_j\vert$ for $1 \leq j \leq k+1$, $g_{j,\ell}$ is the number of edges in $G$ which are directed from node $j$ to node $\ell$, and $g_j = \sum_{\ell} g_{j,\ell}$ is the out-degree of node $j$ for $1 \leq j \leq k+1$.
\end{thm}

\begin{proof}
By the description preceding the theorem, the formula is explained as follows:
\begin{enumerate}
\item[(i)] $G$ is the chosen graph in $\mathcal{K}_{k+1}^{(d)}$ and $\{B_j\}_{j=1}^{k+1}$ is the chosen partition in $\prod_{n,k+1}$.
\item[(ii)] The number of possible phylogenetic trees assigned to the nodes of $G$ is:
\begin{equation}\label{num-pts}
\prod_{j=1}^{k+1}(2b_j-3)!! = \prod_{j=1}^{k+1} \frac{(2b_j-2)!}{2^{b_j-1}(b_j-1)!},
\end{equation}
where we assume that $B_j$ is the block belonging to node $j$ in $G$.
\item[(iii)] The number of ways of adding nodes to the phylogenetic tree of node $j$ is:
\[
(2b_j-1)\cdots(2b_j-1+g_j-1)
\]
\item[(iv)] Connecting the nodes which have been added to the phylogenetic tree of node $j$ to the root of the phylogenetic tree of node $\ell$ gives every tree-child networks $(g_{j,\ell})!$ times. Thus, the number of tree-child networks arising from a fixed choice of $G$, $\{B_j\}_{j=1}^{k+1}$ and a set of phylogenetic trees (whose number of leaves equals to $b_j$ for $1\leq j\leq k+1$) is:
\begin{equation}\label{num-edges}
\prod_{j=1}^{k+1}\frac{(2b_j-1)\cdots(2b_j-1+g_j-1)}{\prod_{\ell=1}^{k+1}(g_{j,\ell})!}
\end{equation}
\item[(v)] Finally, multiplying (\ref{num-pts}) and (\ref{num-edges}) gives
\begin{align*}
\prod_{j=1}^{k+1} \frac{(2b_j-2)!}{2^{b_j-1}(b_j-1)!}
\frac{(2b_j-1)\cdots(2b_j-1+g_j-1)}{\prod_{\ell=1}^{k+1}(g_{j,\ell})!}&=\frac{1}{2^{(\sum_{j}b_j)-k-1}}\prod_{j=1}^{k+1}\,\frac{(2b_j+g_j-2)!}{(b_j-1)! \prod_{\ell=1}^{k+1}(g_{j,\ell})!}\\
&=\frac{1}{2^{n-k-1}}\prod_{j=1}^{k+1}\,\frac{(2b_j+g_j-2)!}{(b_j-1)! \prod_{\ell=1}^{k+1}(g_{j,\ell})!},
\end{align*}
which is the claimed factor (in front of and inside) of the double sum in (\ref{formula-TC-2}).\qedhere
\end{enumerate}
\end{proof}

\subsection{Values and Formulas for Small \texorpdfstring{$k$}{k}}\label{formulas-fixed-k}

Formula \eqref{formula-TC-2} makes it possible to compute the values of $\mathrm{TC}_{n,k}^{(d)}$ for small values of $n,k,d$. (In \cite{CaZh}, this was done for the special case $d=2$.) However, the method presented in Section~\ref{formula-TC} to achieve the same task is less computation-intensive since using \eqref{formula-TC-2} makes it necessary to generate all component graphs with $k+1$ vertices. (The number of component graphs increases rapidly as can be seen from Theorem~\ref{numb-comp-graphs}.)

Another application of (\ref{formula-TC-2}), which was also given in \cite{CaZh}, is the derivation of formulas for small values of $k$ (which hold for all $n$), e.g., in \cite{CaZh} such formulas were obtained for $d=2$ and $k=1$ and $k=2$. (See also \cite{FuGiMa2,PoBa} for related formulas.) Indeed, with (\ref{formula-TC-2}), we can give now generalizations for general $d\geq 2$.

We start with expressions which contain generating functions.

\begin{pro} Set
\[
f_{d}(z):=\sum_{m\geq 1}\frac{(2m+d-2)!}{(m-1)!m!}z^m.
\]
\begin{itemize}
\item[(i)] For $k=1$,
\begin{equation}\label{TCn1}
\mathrm{TC}_{n,1}^{(d)}=\frac{n!}{d!2^{n-2}}[z^n]f_d(z)f_0(z).
\end{equation}
\item[(ii)] For $k=2$,
\begin{equation}\label{TCn2}
\mathrm{TC}_{n,2}^{(d)}=\frac{n!}{d!\,2^{n-3}}\,\sum_{\ell=0}^{d}\,\frac{1}{(d-\ell)!\,\ell!}\,[z^n] f_{2d-\ell}(z)\,f_{\ell}(z)\,f_0(z)-\frac{n!}{(d!)^22^{n-2}}[z^n]\,f_{2d}(z)\,f_{0}^2(z).
\end{equation}
\end{itemize}
\end{pro}

\begin{proof}
We start with $k=1$. By using (\ref{formula-TC-2}), we have
\[
\mathrm{TC}_{n,1}^{(d)}=\frac{1}{2^{n-2}}\,\sum_{\{B_j\}_{j=1}^{2}\in\prod_{n,2}}\,\sum_{G\in\mathcal{K}_{2}^{(d)}}\,\prod_{j=1}^{2}\,\frac{(2b_j+g_j-2)!}{(b_j-1)!\,\prod_{\ell=1}^{2}(g_{j,\ell})!}.
\]
Observe that $\mathcal{K}_{2}^{(d)}$ contains only two graphs, namely, the graph consisting of a root to which a child is attached by $d$ edges and either the root has label $1$ or $2$. Consequently,
\[
\mathrm{TC}_{n,1}^{(d)}=\frac{1}{d!\,2^{n-2}}\,\sum_{\{B_j\}_{j=1}^{2}\in\prod_{n,2}}\,\left(\frac{(2b_1+d-2)!}{(b_1-1)!}\cdot\frac{(2b_2-2)!}{(b_2-1)!}+ \frac{(2b_1-2)!}{(b_1-1)!}\cdot\frac{(2b_2+d-2)!}{(b_2-1)!}\right).
\]
Summing according to the size of the blocks in the partition $\{B_j\}_{j=1}^{2}$, we have
\begin{align*}
\mathrm{TC}_{n,1}^{(d)}&=\frac{1}{d!\,2^{n-2}}\,\sum_{\substack{b_1 + b_2 = n,\\ b_1,b_2\geq 1}}\,\binom{n-1}{b_1-1}\,\frac{(2b_1+d-2)!\,(2b_2-2)!+(2b_1-2)!\,(2b_2+d-2)!}{(b_1-1)!\,(b_2-1)!} \\
&=\frac{1}{d!\,2^{n-2}}\,\sum_{b=1}^{n-1}\,\binom{n}{b}\,\frac{(2b+d-2)!\,(2n-2b-2)!}{(b-1)!\,(n-b-1)!}\end{align*}
which translates into (\ref{TCn1}) by using the generating function $f_d(z)$.

Next, we consider $k=2$. Here, again by (\ref{formula-TC-2}),
\[
\mathrm{TC}_{n,2}^{(d)}=\frac{1}{2^{n-3}}\,\sum_{\{B_j\}_{j=1}^{3}\in\prod_{n,3}}\,\sum_{G\in\mathcal{K}_{3}^{(d)}}\,\prod_{j=1}^{3}\,\frac{(2b_j+g_j-2)!}{(b_j-1)!\,\prod_{\ell=1}^{3}(g_{j,\ell})!}.
\]
In contrast to $k=1$, there are now more possibilities for $G$: $G$ has three vertices $v_1,v_2,v_3$ one of which is the root (say $v_1$); $v_1$ is connected to $v_2$ by $d$ edges and to $v_3$ by $\ell$ edges with $0\leq\ell\leq d$; moreover, $v_2$ is connected to $v_3$ by $d-\ell$ edges; finally, there are $3$ possible labelings if $\ell=d$ and $3!=6$ possible labelings if $\ell<d$. Thus,
\begin{align*}
\mathrm{TC}_{n,2}^{(d)}=&\frac{1}{2^{n-3}}\sum_{b_1+b_2+b_3=n}\binom{n}{b_1,b_2,b_3}\frac{1}{3!}\\
&\qquad\times\left(\sum_{\ell=0}^{d-1}\frac{3!}{d!\ell!(d-\ell)!}\cdot\frac{(2b_1+d+\ell-2)!}{(b_1-1)!}\cdot\frac{(2b_2+d-\ell-2)!}{(b_2-1)!}\cdot\frac{(2b_3-2)!}{(b_3-1)!}\right.\\
&\qquad\qquad\left.+\frac{3}{(d!)^2}\cdot\frac{(2b_1+2d-2)!}{(b_1-1)!}\cdot\frac{(2b_2-2)!}{(b_2-1)!}\cdot\frac{(2b_3-2)!}{(b_3-1)!}\right).
\end{align*}
From this, by using the generating $f_d(z)$, we obtain
\[
\mathrm{TC}_{n,2}^{(d)}=\frac{n!}{d!2^{n-3}}\sum_{\ell=1}^{d}\frac{1}{\ell!(d-\ell)!}[z^n]f_{2d-\ell}(z)f_{\ell}(z)f_0(z)+\frac{n!}{(d!)^22^{n-2}}[z^n]f_{2d}(z)f_0^2(z)
\]
which is equivalent to (\ref{TCn2}).
\end{proof}

In order to simplify these expressions, we need two technical lemmas. The first gives a recursive way of computing $f_d(z)$.

\begin{lmm}\label{rec-fdX}
Set $X:=\sqrt{1-4z}$. Then, $f_d(X)$ can be recursively computed as
\[
f_d(X)=(-X^{-1}+X)f_{d-1}'(X)+(d-2)f_{d-1}(X),\qquad(d\geq 1)
\]
with initial condition $f_0(X)=1/2-X/2$.
\end{lmm}

\begin{proof}
First, for the initial condition, observe that
\[
\frac{(2m-2)!}{(m-1)!m!}=\frac{1}{m}\binom{2m-2}{m-1}
\]
and thus $f_0(z)$ is the generating function of the (shifted) Catalan numbers. Consequently,
\[
f_0(z)=\frac{1-\sqrt{1-4z}}{2}=\frac{1-X}{2}.
\]

Next, in order to prove the recurrence, we have
\begin{align*}
f_d(z)&=2\sum_{m\geq 1}\frac{m(2m+d-3)!}{(m-1)!m!}z^m+(d-2)\sum_{m\geq 1}\frac{(2m+d-3)!}{(m-1)!m!}z^m\\
&=2zf_{d-1}'(z)+(d-2)f_{d-1}'(z).
\end{align*}
Writing this in terms of $X$ gives the claimed result.
\end{proof}

From this lemma, we see that the first few values of $f_d(X)$ are:
\[
f_1(X)=\frac{1}{2X}-\frac{1}{2},\quad f_2(X)=\frac{1}{2X^3}-\frac{1}{2X},\quad f_3(X)=\frac{3}{2X^5}-\frac{3}{2X^3}.
\]
In particular, it is not hard to see that $f_d(X)$ for $d\geq 2$ is a (finite) linear combinations of terms of the form $X^{-m}$ with $m$ odd.

We need a second technical lemma that helps us with the extraction of coefficients.

\begin{lmm}
Let $X:=\sqrt{1-4z}$. Then, for odd $m$
\[
[z^n]\,X^{-m}=\frac{1}{\binom{m-1}{(m-1)/2}}\,\binom{n+(m-1)/2}{(m-1)/2}\,\binom{2n+m-1}{n+(m-1)/2}
\]
and for even $m$
\[
[z^n]\,X^{-m}=4^n\,\binom{n+(m-2)/2}{(m-2)/2}.
\]
\end{lmm}
\begin{proof}
Note that
\[
[z^n]\,X^{-m}=\binom{-m/2}{n}\,(-4)^n= \frac{m(m+2)\cdots(m+2n-2)}{n!}\,2^n.
\]
From this both claims follow by standard manipulations.
\end{proof}

Now, we can simplify (\ref{TCn1}) for small values of $d$;
recall the notation~\eqref{eq:doublfac} of double factorials.

\begin{cor}
The number of tree-child networks with $n$ leaves and $1$ reticulation node is
\begin{itemize}
\item[(i)] for $d=2$:
\[
\mathrm{TC}_{n,1}^{(2)}=n \big((2n-1)!!-(2n-2)!! \big);
\]
\item[(ii)] for $d=3$:
\[
\mathrm{TC}_{n,1}^{(3)}=\frac{n(2n+1)}{3}(2n-1)!!-n^2(2n-2)!! \, .
\]
\end{itemize}
\begin{proof}
First, from (\ref{TCn1}) and the initial condition in Lemma~\ref{rec-fdX}:
\begin{equation}\label{st-1}
\mathrm{TC}_{n,1}^{(d)}=\frac{n!}{d!2^{n-1}}[z^n]f_d(z)-\frac{n!}{d!2^{n-1}}[z^n]f_d(z)X=\frac{(2n+d-2)!}{d!2^{n-1}(n-1)!}-\frac{n!}{d!2^{n-1}}[z^n]f_d(z)X.
\end{equation}
The second term becomes for $d=2$,
\[
[z^n]f_2(z)X=[z^{n}]\left(\frac{1}{2\,X}- \frac{1}{2\,X^3}\right)X=-\frac{1}{2}[z^{n}]X^{-2}=-\frac{4^n}{2}
\]
and for $d=3$,
\[
[z^n]f_3(z)X=[z^{n}]\left(\frac{3}{2X^5}-\frac{3}{2X^3}\right)X=\frac{3}{2}[z^n]X^{-4}-\frac{3}{2}[z^n]X^{-2}
=\frac{3\cdot 4^n}{2}n.
\]
Plugging this into (\ref{st-1}) and standard manipulations give the claimed result.
\end{proof}
\end{cor}
\begin{rem}
The formula for $d=2$ is known; see, e.g., \cite{CaZh,FuGiMa2,PoBa}.
\end{rem}

\begin{rem}
The method of proof gives the following structural result for general $d$ and $n \geq 2$:
\begin{align*}
\mathrm{TC}_{n,1}^{(d)}=\binom{2n+d-2}{d} (2n-3)!! - p_d(n) (2n-2)!!,
\end{align*}
where $p_d(n)$ is a polynomial of degree $d-1$.
Note that $\mathrm{TC}_{1,1}^{(d)}=0$.
\end{rem}

Likewise, we can also simplify (\ref{TCn2}) for small values of $d$.

\begin{cor}
The number of tree-child networks with $n$ leaves and $2$ reticulation nodes is
\begin{itemize}
\item[(i)] for $d=2$:
\begin{align*}
\mathrm{TC}_{n,2}^{(2)}=n(n-1)\left( \frac{3n+2}{3}\,(2n-1)!! - (2n)!! \right);
\end{align*}
\item[(ii)] for $d=3$:
\begin{align*}
\mathrm{TC}_{n,2}^{(3)}= n(n-1)\left( \frac{70n^2+244n+177}{315}\,(2n+1)!! - \frac{16n+13}{48}\,(2n+2)!! \right).
\end{align*}
\end{itemize}
\end{cor}
\begin{rem}
The formula for $d=2$ is again known; see, e.g., \cite{CaZh,FuGiMa2,PoBa}.
\end{rem}

\subsection{Asymptotics for Fixed \texorpdfstring{$k$}{k}}\label{asymp-fixed-k}

In this subsection, we give a final application of the method of component graphs, namely, we derive the first-order asymptotics of $\mathrm{TC}_{n,k}^{(d)}$ for fixed $k$ as $n\rightarrow\infty$. This extends the main result from \cite{FuHuYu} to general $d$.

The main observation of \cite{FuHuYu} was that the main asymptotic contribution to the asymptotics of $\mathrm{TC}_{n,k}^{(d)}$ comes from the tree-child networks constructed from the star-component graph with $k$ leaves; see Figure~\ref{star-component}. (In fact, since component graphs are labeled, there are $k+1$ star-component graphs depending on the label of the root vertex.)

\begin{figure}[t]
\centering
\includegraphics[scale=1.3]{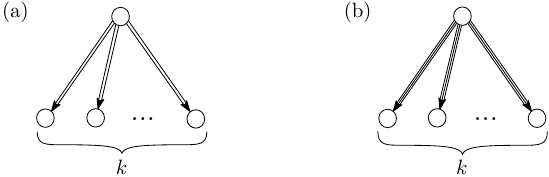}
\caption{Star-component graphs for generating the tree-child networks whose numbers dominate the asymptotics of $\mathrm{TC}_{n,k}^{(d)}$ for fixed $k$ as $n$ tends to infinity. (Left: $d=2$; Right: $d=3$. Labels of nodes are removed.)}\label{star-component}
\end{figure}

We denote by $S_{n,k}^{(d)}$ the number of tree-child networks arising from the star-component graph(s) with $k$ leaves. Then, we have the following formula for this number which generalizes the formula from \cite[Lemma~5]{FuHuYu} for $d=2$.

\begin{lmm}\label{Snk-d}
We have,
\[
S_{n,k}^{(d)}=\frac{n!}{(d!)^{k}2^{n-k-1}(k-1)!}\sum_{j=1}^{n-k}\frac{(2j+dk-2)!}{j!(j-1)!}\cdot\frac{(2n-k-2j-1)!}{(n-k-j)!(n-j)!}.
\]
\end{lmm}
\begin{proof}
The proof is very similar to the one of \cite[Lemma~5]{FuHuYu}. We give a short sketch.

First, by a combinatorial argument, we have
\[
S_{n,k}^{(d)}=\sum_{j=1}^{n-k}\binom{n}{j}\frac{(2j+dk-2)!}{(d!)^k2^{j-1}(j-1)!}\cdot\frac{1}{k!}(n-j)![z^{n-j}]T(z)^k,
\]
where
\[
T(z)=\sum_{n\geq 1}(2n-1)!!\cdot\frac{z^n}{n!}=1-\sqrt{1-2z}
\]
is the exponential generating function of the number of phylogenetic trees.

Briefly, the construction on which this combinatorial argument is based works as follows: first, we pick a one-component tree-child network with $j+k$ leaves where the leaves with labels $\{1,\ldots,k\}$ are the ones below the reticulation nodes. Then, we replace the leaves below reticulation nodes by phylogenetic trees. (For this, we need a forest of $k$ phylogenetic trees.) Finally, we re-label the leaves.

Next, by a standard application of the Lagrange inversion formula, we obtain
\[
[z^{n-j}]T(z)^k=\frac{k}{n-j}2^{j+k-n}\binom{2n-k-2j-1}{n-k-j}.
\]
Plugging this into the expression above and straightforward manipulations yield the claimed result.
\end{proof}

Applying to this the Laplace method, we have the following asymptotic result.

\begin{pro}
For fixed $k$, as $n\rightarrow\infty$,
\[
S^{(d)}_{n,k}\sim\frac{2^{dk-1}}{(d!)^{k}\,k!\sqrt{\pi}}n!2^nn^{dk-3/2}.
\]
\end{pro}
\begin{proof}
The proof is similar to \cite[Lemma~6]{FuHuYu}. Again, we just give a sketch.

First, it suffices to consider the asymptotics of the sum in the expression for $S_{n,k}^{(d)}$ from Lemma~\ref{Snk-d} as the factor in front of it has already the right shape. Thus, we set:
\begin{align*}
\Sigma_{n,k}&:=\sum_{j=1}^{n-k}\frac{(2j+dk-2)!}{j!(j-1)!}\cdot\frac{(2n-k-2j-1)!}{(n-k-j)!(n-j)!}\\
&=\sum_{j=0}^{n-k-1}\frac{(2n+(d-2)k-2j-2)!}{(n-k-j)!(n-k-j-1)!}\cdot\frac{(2j+k-1)!}{j!(j+k)!}.
\end{align*}

Now, observe that the first term in the last sum is decreasing in $j$ and has the expansion:
\[
\frac{(2n+(d-2)k-2j-2)!}{(n-k-j)!(n-k-j-1)!}=\frac{2^{(d-2)k}}{\sqrt{\pi}}4^{n-j-1}n^{dk-3/2}\left(1+\mathcal{O}\left(\frac{1+j}{n}\right)\right)
\]
uniformly in $j$ as $j=o(n)$. Thus, by a standard application of the Laplace method:
\[
\Sigma_{n,k}\sim\frac{2^{(d-2)k}}{\sqrt{\pi}}\left(\sum_{j=0}^{\infty}\frac{(2j+k-1)!}{ j!(j+k)!}4^{-j}\right)4^{n-1}n^{dk-3/2}=\frac{2^{(d-1)k}}{k\sqrt{\pi}}4^{n-1}n^{dk-3/2},
\]
where we used \cite[Lemma~7]{FuHuYu} in the last step.

From the above asymptotics multiplied with the factor in front of the sum in the formula for $S_{n,k}^{(d)}$ from Lemma~\ref{Snk-d}, we obtain the claimed result.
\end{proof}

Finally, using the same arguments as in \cite{FuHuYu}, we can show that also here the contribution from tree-child networks arising from the star-component graph(s) dominates; we leave details to the reader.

\begin{thm}
For the number of $d$-combining tree-child networks with $n$ leaves and $k$ reticulation nodes, we have for fixed $k$, as $n\rightarrow\infty$,
\[
\mathrm{TC}^{(d)}_{n,k} \sim \frac{2^{dk-1}}{(d!)^{k}\,k!\,\sqrt{\pi}}\,n!\,2^n\,n^{dk-3/2}.
\]
\end{thm}
\begin{rem}
This result was stated in \cite[Theorem~8]{ChFuLiWaYu} without proof. (Note that this also corrects two typos in the statement of \cite[Theorem~8]{ChFuLiWaYu}.)
\end{rem}

\section{Proofs of Propositions~\ref{lem:AiryXLower} and \ref{lem:AiryXUpper}}\label{App-C}

    The proofs follow nearly verbatim the steps of \cite[Lemmas~4.2 and 4.4]{ElFaWa} as well as \cite[Lemmas~7 and 9]{ElFaWa2020}.
    For the convenience of the reader, we repeat the main steps for Proposition~\ref{lem:AiryXLower} and point out the main differences.
    Full details of all computations are given in our accompanying Maple worksheet~\cite{Wa}.

    We start by defining the following sequence
    \begin{align*}
        P_{n,m} :=
            -{Z}_{n,m}{s}_{n}
            +\mu_{n,m}^{(d)} {Z}_{n-1,m+1}
            + \nu_{n,m}^{(d)} {Z}_{n-1,m-1},
    \end{align*}
    where we use the ansatz
    \begin{align*}
          s_n&:= \sigma_0 + \frac{\sigma_1}{n^{1/3}} + \frac{\sigma_2}{n^{2/3}}+\frac{\sigma_3}{n} + \frac{\sigma_4}{n^{7/6}}\\
           Z_{n,m}&:= \left(1+\frac{\tau_2 m^2+\tau_1 m}{n}\right)\Ai\left(a_{1}+\frac{\bb^{1/3}(m+1)}{n^{1/3}}\right),
    \end{align*}
    with parameters $\sigma_i, \tau_j \in\mathbb{R}$, and $\bb$ from~\eqref{eq:fansatz}.
    Note that in \cite{ElFaWa,ElFaWa2020} we had $B=2$, whereas here $B$ depends on $d$.
	Then the claimed inequality is equivalent to $P_{n,m} \geq 0$ with the parameters chosen accordingly.
	
	To prove it, we expand the Airy function $\Ai(z)$ in a neighborhood of
	\begin{align}
		\label{eq:Pnmexpansionlow}
		\aiarg = a_{1}+\frac{\bb^{1/3} m}{n^{1/3}},
	\end{align}
	and, due to the defining differential equation $\Ai''(x) = x \Ai(x)$, we get the following expansion
	\begin{align*}
		P_{n,m} &=  p_{n,m} \Ai(\aiarg) + p'_{n,m} \Ai'(\aiarg),
	\end{align*}
	where $p_{n,m}$ and $p'_{n,m}$ are functions of $m$ and $n^{-1}$ and may be expanded as power series in $n^{-1/6}$ whose coefficients are polynomials in $m$.
	As the Airy function is entire, these series converge absolutely for $n>1$ and $m<n$.
	
	Now we proceed with the technical analysis.
	First, we show that $[m^i n^{j}]P_{n,m} = 0$ for $i+j > 1$, $i,j \in \mathbb{Q}$.
	We omit this step, as it is analogous to the previous cases.
	Second, we use computer algebra to strengthen this result by choosing suitable values $\sigma_i$ and $\tau_j$ to eliminate more terms; see Figure~\ref{fig:PosLowerBound1}.
	A solid diamond at $(i,j)$ indicates that the coefficient $[m^in^j]P_{n,m}$ is non-zero for generic values of $\sigma_i$ and $\tau_j$;
	an empty diamond indicates that the specific choice of $\sigma_i$ and $\tau_j$ makes it vanish.
	The convex hull is formed by the following three lines
	\begin{align*}
		L_1 &: j = -\frac{7}{6} - \frac{7i}{18}, &
		L_2 &: j = -\frac{1}{3} - \frac{2i}{3}, &
		L_3 &: j = 1 - i.
	\end{align*}
	
	\begin{figure}
		\centering
		\includegraphics[width=0.48\textwidth]{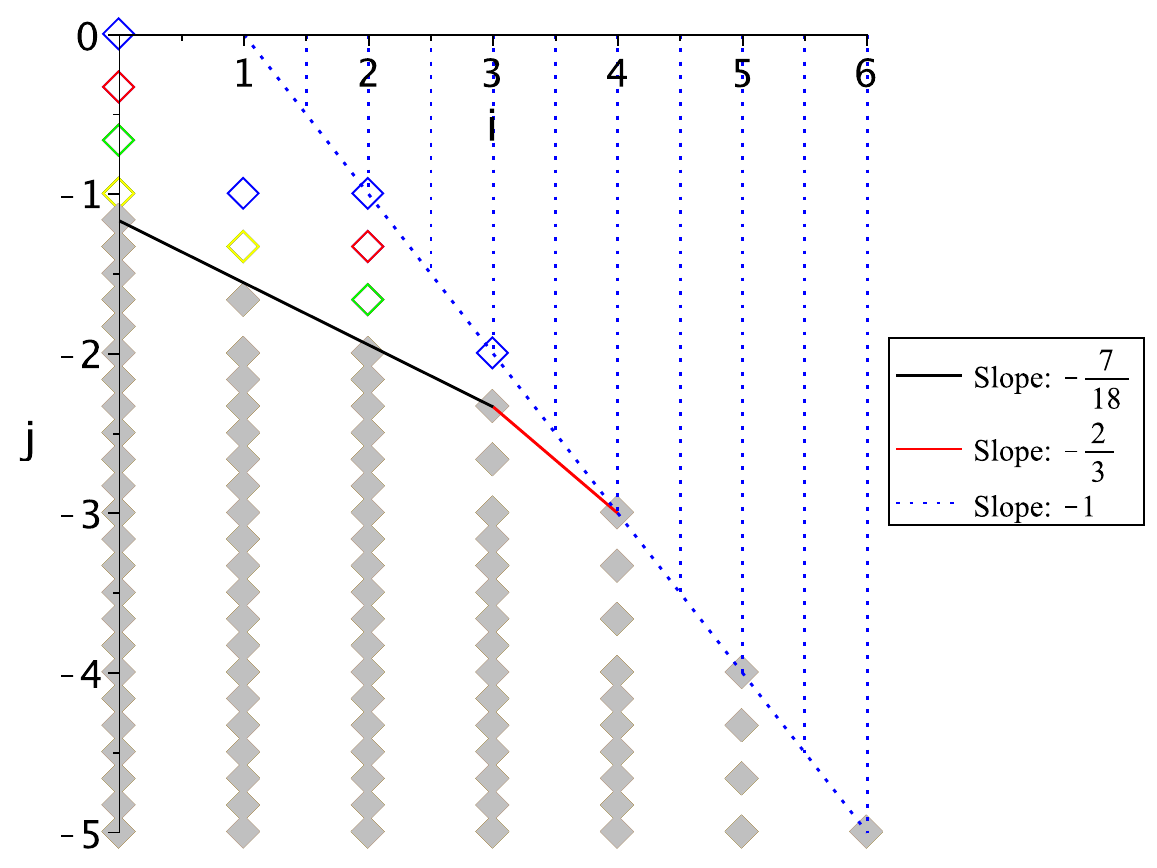}%
		\quad
		\includegraphics[width=0.48\textwidth]{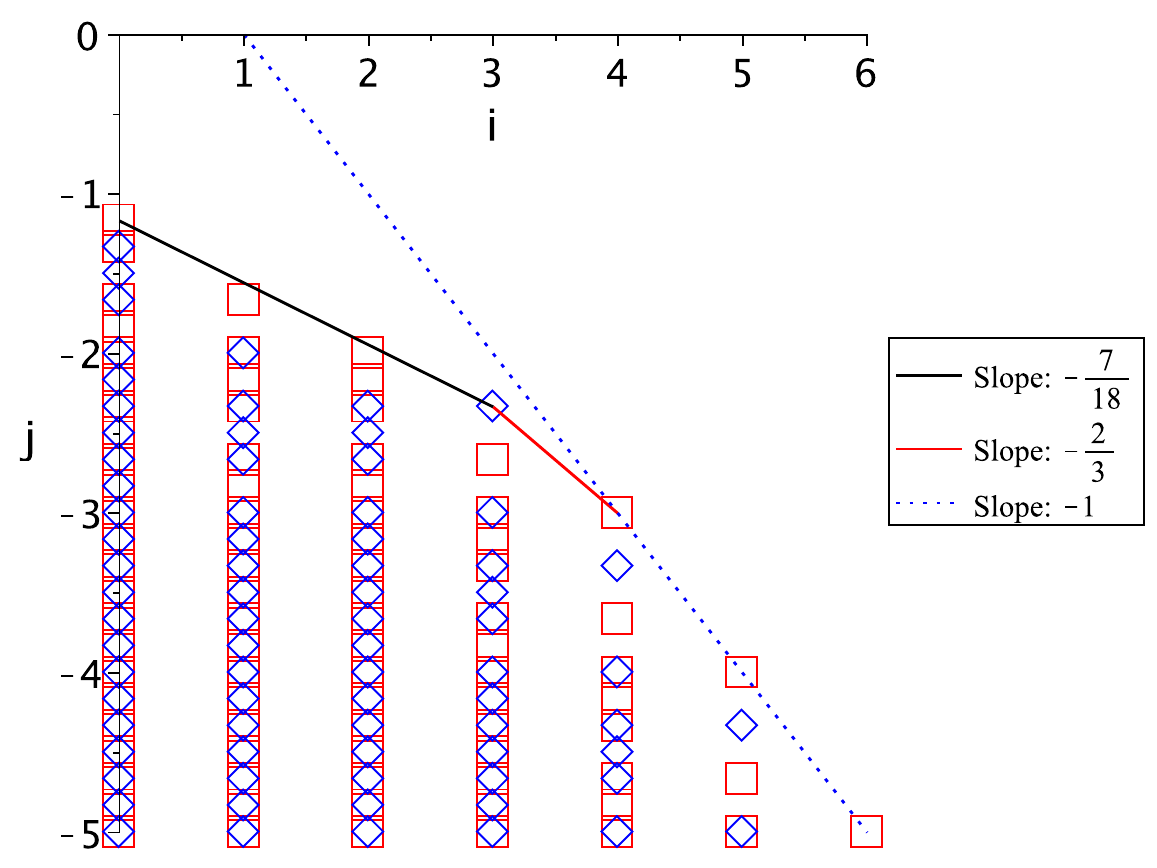}%
		\caption{(Left) Non-zero coefficients of $P_{n,m} = \sum {a_{i,j}} m^i n^j$ shown by diamonds for $s_n := \sigma_0 + \frac{\sigma_1}{n^{1/3}} + \frac{\sigma_2}{n^{2/3}}+\frac{\sigma_3}{n} + \frac{\sigma_4}{n^{7/6}}$ and $Z_{n,m} := \left(1+\frac{\tau_2 m^2+\tau_1 m}{n}\right)\Ai\left(a_{1}+\frac{2^{1/3}(m+1)}{n^{1/3}}\right)$. There are no terms in the blue dashed area. The blue terms vanish for $\sigma_0=2$, the red terms vanish for $\sigma_1=0$, the green terms vanish for $\sigma_2=B^{2/3}a_1$, and the yellow terms vanish for
		$\sigma_3=-\frac{3d^2-5d+4}{3(d+1)}$ and $\tau_2=-\frac{2d-1}{3(d+1)}$.
		The black, red, and blue lines represent the parts $L_1$, $L_2$, and $L_3$, respectively, of the convex hull.
		(Right) The solid gray diamonds are decomposed into the coefficients $p_{n,m}$ of $\Ai(\alpha)$ (red boxes) and $p'_{n,m}$ of $\Ai'(\alpha)$ (blue diamonds).}
		\label{fig:PosLowerBound1}
	\end{figure}
	
	In a final step, we distinguish between $p_{n,m}$ and $p'_{n,m}$; see Figure~\ref{fig:PosLowerBound2}.
	The expansions for $n$ tending to infinity start as follows, where the elements on the convex hull are written in color:
	\begin{align*}
		&P_{n,m} =\Ai(\alpha) \left(
			\textcolor{red}{-\frac{\sigma_4}{n^{7/6}}}
			- \frac{B^{5/3} a_1 m}{3 n^{5/3}}
			\textcolor{red}{- \frac{(23d^2-14d+5) m^2}{9(d+1)^2 n^2}}
			- \frac{2(2d-1)(3d-1) B^{5/3} a_1 m^3}{9 (d+1)^2 n^{8/3}}
			\right. \\
			& \left. \qquad \qquad \qquad \quad
			\textcolor{red}{- \frac{(2d-1)(23d-9) B m^4}{18(d+1)^2 n^3}
			+ \frac{(2d - 1)(209d^2 - 258d + 129) B m^5}{270 (d+1)^3 n^4}}
			+  \ldots
					\right) + \\
			    &~\Ai'(\alpha) \left(\!
			\textcolor{blue}{- \frac{B^{1/3}\sigma_4}{n^{3/2}}}
			\!-\! \frac{4 (d - 2) B a_1 m}{9(d+1) n^2}
			\!-\! \frac{2(9d^3 + 50d^2 - 67d + 21) B^{1/3} m^2}{9 (d+1)^2 n^{7/3}}\textcolor{blue}{- \frac{ 4(2d-1)^2 B^{1/3} m^3} {9(d+1)^2 n^{7/3}}}
			 \right.\\
			& \left. \qquad \qquad \textcolor{blue}{- \frac{ (2d-1)(5d-1) B^{4/3} m^4}{18(d+1)^2 n^{10/3}}}\textcolor{blue}{- \frac{(2d-1)(119d^2+32d-51)B^{4/3} m^5}{270 (d+1)^3 n^{13/3}}}
			+ \ldots
			\right).
	\end{align*}	
	
	\begin{figure}
		\centering
		\includegraphics[width=0.48\textwidth]{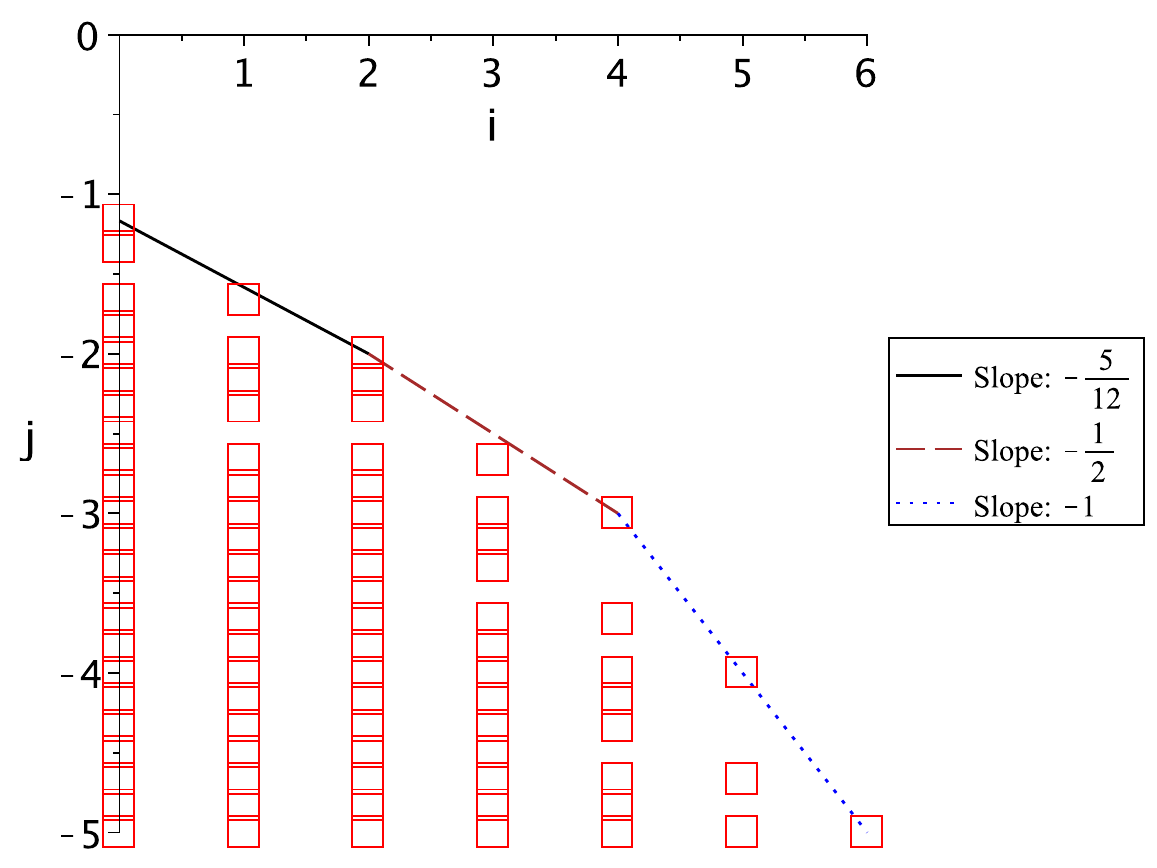}%
		\quad
		\includegraphics[width=0.48\textwidth]{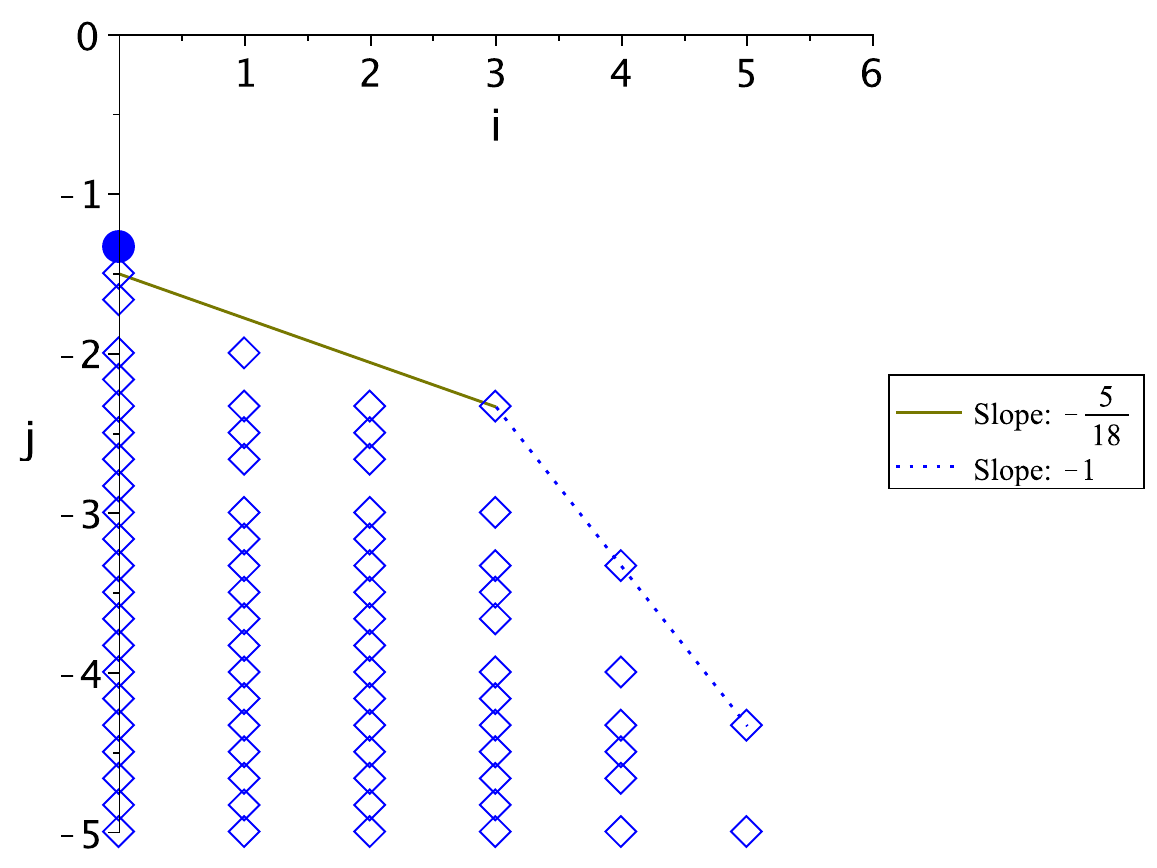}%
		\caption{Non-zero coefficients $p_{n,m} = \sum \tilde{a}_{i,j} m^i n^j$ (red) and $p'_{n,m} = \sum \tilde{a}_{i,j}' m^i n^j$ (blue) of the expansion~\eqref{eq:Pnmexpansionlow} for $P_{n,m}$. The coefficient of $n^{-4/3}$ in the right picture depicted as a solid blue circle disappears for $\tau_1=\frac{3d^2+12d-11}{6(d+1)}$.}
		\label{fig:PosLowerBound2}
	\end{figure}
	
	We now choose $\sigma_4=-1$ which leads to a positive term $\Ai(\alpha)n^{-7/6}$.
	Next, for fixed (large) $n$ we prove that for all $m$ the dominant contributions in $P_{n,m}$ are positive.
	Motivated by Figures~\ref{fig:PosLowerBound1} and \ref{fig:PosLowerBound2}, we consider three different regimes:
	\[
	\text{(i)}\quad m \leq C n^{1/3};\qquad\text{(ii)}\quad C n^{1/3} < m \leq n^{7/18};\qquad\text{(iii)}\quad n^{7/18} < m <n^{2/3-\epsilon}
	\]
	for a suitable constant $C>0$.
	This part is analogous to the one of \cite[Lemma~4.2]{ElFaWa}, which is why we omit the technical details.
	In the end we get that there exists an $N>0$ such that all terms are positive for $n > N$ and all $m < n^{2/3}$, which ends the proof of Proposition~\ref{lem:AiryXLower}.
	
	The proof of Proposition~\ref{lem:AiryXUpper} follows analogously.
 \qedhere


{
\small
\begin{thebibliography}{99}
\bibitem{BaMaWa} C. Banderier, P. Marchal, M. Wallner (2018). Rectangular Young tableaux with local decreases and the density method for uniform random generation, {\it GASCom 2018}, CEUR Workshop
Proceedings. Vol. 2113. 2018, pp. 60--68.
\bibitem{BaWa} C. Banderier and M. Wallner (2021). Young tableaux with periodic walls: counting with the density method, {\it Sém. Lothar. Combin.} 85B, Art. 47, 12 pp.
\bibitem{CaZh} G. Cardona and L. Zhang (2020). Counting and enumerating tree-child networks and their subclasses, {\it J. Comput. System Sci.}, {\bf 114}, 84--104.
\bibitem{ChFuLiWaYu} Y.-S. Chang, M. Fuchs, H. Liu, M. Wallner, G.-R. Yu (2022). Enumeration of $d$-combining tree-child networks, LIPICS, Proceedings of the 33rd Meeting on Probabilistic, Combinatorial and Asymptotic Methods for the Analysis of Algorithms, 225, Paper 5.
\bibitem{ElFaWa2020} A. Elvey Price, W. Fang, M. Wallner (2020). Asymptotics of minimal deterministic finite automata recognizing a finite binary language, LIPICS, Proceedings of the 31st Meeting on Probabilistic, Combinatorial and Asymptotic Methods for the Analysis of Algorithms, 159, Paper~11.
\bibitem{ElFaWa} A. Elvey Price, W. Fang, M. Wallner (2021). Compacted binary trees admit a stretched exponential, {\it J. Comb. Theory Ser. A}, {\bf 177}, Article 105306.
\bibitem{FiHeKeKuWi} M. Fischer, L. Herbst, S. Kersting, L. K\"{u}hn, K. Wicke. Tree balance
indices: a comprehensive survey, arXiv:2109.12281.
\bibitem{FiHeKeKuWi2} M. Fischer, L. Herbst, S. Kersting, L. K\"{u}hn, K. Wicke. {\it Tree Balance
Indices: A Comprehensive Survey}, Springer, 1st edition, 2023.
\bibitem{FlSe} P. Flajolet and R. Sedgewick. {\it An Introduction to the Analysis of Algorithms}, Addison-Wesley Professional, 2nd edition, 2013.
\bibitem{FuGiMa1} M. Fuchs, B. Gittenberger, M. Mansouri (2019). Counting phylogenetic networks with few reticulation vertices: tree-child and normal networks, {\it Australas. J. Combin.},{\bf 73:2}, 385--423.
\bibitem{FuGiMa2} M. Fuchs, B. Gittenberger, M. Mansouri (2021). Counting phylogenetic networks with few reticulation vertices: exact enumeration and corrections, {\it Australas. J. Combin.}, {\bf 82:2}, 257--282.
\bibitem{FuHuYu} M. Fuchs, E.-Y. Huang, G.-R. Yu (2022). Counting phylogenetic networks with few reticulation vertices: a second approach, {\it Discrete Appl. Math.}, {\bf 320}, 140--149.
\bibitem{FuLiYu} M. Fuchs, H. Liu, G.-R. Yu. A short note on the exact counting of tree-child networks, arXiv:2110.03842.
\bibitem{FuYuZh} M. Fuchs, G.-R. Yu, L. Zhang (2021). On the asymptotic growth of the number of tree-child networks, {\it European J. Combin.}, {\bf 93}, 103278, 20 pp.
\bibitem{HuRuSc} D. H. Huson, R. Rupp, C. Scornavacca. {\it Phylogenetic Networks: Concepts, Algorithms and Applications}, Cambridge University Press, 1st edition, 2010.
\bibitem{DiSeWe} C. McDiarmid, C. Semple, D. Welsh (2015). Counting phylogenetic networks, {\it Ann. Comb.}, {\bf 19:1},
205--224.
\bibitem{PoBa} M. Pons and J. Batle (2021). Combinatorial characterization of a certain class of words and a
conjectured connection with general subclasses of phylogenetic tree-child networks, {\bf 11}, Article number: 21875.
\bibitem{St} M. Steel. {\it Phylogeny—Discrete and Random Processes in Evolution}, CBMS-NSF Regional Conference Series in Applied Mathematics, 89, Society for Industrial and Applied Mathematics (SIAM),
Philadelphia, PA, 2016.
\bibitem{Wa} M. Wallner, \href{ 	
https://doi.org/10.48550/arXiv.2203.07619}{Ancilliary files of arXiv:2203.07619} and \url{http://dmg.tuwien.ac.at/mwallner}, 2022.
\bibitem{Zh2} L. Zhang (2019). Generating normal networks via leaf insertion and nearest neighbor interchange, {\it BMC Bioinformatics}, {\bf 20:20}, 1--9.
\bibitem{Zh} L. Zhang (2022). The Sackin index of simplex networks, In: Jin, L., Durand, D. (eds) Comparative Genomics. RECOMB-CG 2022, Lecture Notes in Computer Science, {\bf 13234}, Springer, Cham.
\end{thebibliography}
}
\end{document}